\documentclass[a4paper,11pt,twoside,reqno]{article}

\usepackage{hyperref}
\usepackage[english]{babel}
\usepackage{fancyhdr}
\usepackage{fullpage}
\usepackage[dvips]{graphicx}
\usepackage[T1]{fontenc}
\usepackage{fouriernc}
\usepackage[latin1,applemac]{inputenc}
\usepackage{amsmath}
\usepackage{amsfonts}
\usepackage{amssymb}
\usepackage{mathrsfs}
\usepackage{amsthm}
\usepackage{latexsym}
\usepackage{array}
\usepackage{enumerate}
\usepackage{mathrsfs}
\usepackage{amsxtra}
\usepackage{amscd}
\usepackage{mathtools}
\usepackage[usenames]{color}
\definecolor{red}{rgb}{1.0,0.0,0.0}

\definecolor{blu}{rgb}{0.0,0.0,1.0}

\definecolor{violet}{rgb}{0.5,0.0,0.4}

\usepackage[usenames]{color}

\newtheorem{Theorem}{Theorem}[section]
\newtheorem{Definition}[Theorem]{Definition}
\newtheorem{Proposition}[Theorem]{Proposition}
\newtheorem{Lemma}[Theorem]{Lemma}
\newtheorem{Corollary}[Theorem]{Corollary}
\newtheorem{Assumption}[Theorem]{Assumption}

\newtheorem{Remark}[Theorem]{Remark}
\newtheorem{Example}[Theorem]{Example}

\numberwithin{equation}{section}

\newcommand{\myref}[1]{(\ref {#1})}
\def\cali{\mathcal{I}}

\def\R{\mathbb R}
\def\N{\mathbb N}
\def\E{\mathbb E}

\def\cald{{\mathcal D}}
\def\cala{{\mathcal A}}

\def\calg{{\mathcal G}}

\def\call{{\mathcal L}}
\def\caln{{\mathcal N}}
\def\calp{{\mathcal P}}

\def\<{\left\langle }
\def\>{\right\rangle }

\def\Swiech
{{\accent1 S}wie{\hbox{\kern -0.21em\lower
0.77ex\hbox{$\textfont1=\scriptfont1
\lhook$}}}ch}

\def\SWIECH
{{\accent1 S}WIE{\hbox{\kern -0.21em\lower
0.77ex\hbox{$\textfont1=\scriptfont1
\lhook$}}}CH}

\newtheoremstyle{mytheorem}
{6pt}
{6pt}
{\itshape}
{-0pt}
{\large \scshape}
{}
{1em}
{}

\newtheoremstyle{myremark}
{6pt}
{10pt}
{\rm}
{-0pt}
{\large \scshape}
{}
{1em}
{}

\makeatletter
\def\@endtheorem{{\hfill\hbox{\enspace${ \blacksquare}$}}\endtrivlist\@endpefalse } 
\makeatother

\def\ets{e^{(t-s)A}}

\def\calf{{\mathcal{F}}}
\def\calg{{\mathcal{G}}}
\def\calp{{\mathcal{P}}}
\def\calr{{\mathcal{R}}}

\def\P{\mathfrak{P}}

\def\ets{e^{(t-s)A}}

\def\N{\mathcal{N}}

\def\norm{{\| \kern -.05em | }}

\newcommand{\reals}{{{\rm I} \kern -.15em {\rm R} }}
\newcommand{\nat}{{{\rm I} \kern -.15em {\rm N} }}

\newcommand{\ud}{{d}}

\def\R{\mathbb R}
\def\N{\mathbb N}

\def\E{\mathbb E}
\def\P{\mathbb P}

\def\cala{{\mathcal A}}

\def\cald{{\mathcal D}}

\def\calf{{\mathcal F}}
\def\calg{{\mathcal G}}
\def\calh{{\mathcal H}}
\def\cali{{\mathcal I}}

\def\call{{\mathcal L}}

\def\caln{{\mathcal N}}
\def\calp{{\mathcal P}}

\def\calu{{\mathcal U}}

\def\calo{{\mathcal O}}

\def\to{\rightarrow}

\def\<{\left\langle }
\def\>{\right\rangle }

\begin{document}

\title{Mild solutions of semilinear elliptic equations\\ in Hilbert spaces}
\author{Salvatore Federico\footnote{Dipartimento di Economia Politica e Statistica, Universit\`a degli Studi di Siena, Siena (Italy). Email: \texttt{salvatore.federico@unisi.it}.} \and Fausto Gozzi\footnote{Dipartimento di Economia e Finanza, LUISS University, Rome (Italy). Email: \texttt{fgozzi@luiss.it}. }}

\maketitle

\begin{abstract}
This paper extends the theory of regular solutions ($C^1$ in a suitable sense) for
a class of semilinear elliptic equations in Hilbert spaces.
The notion of regularity is based on the concept of $G$-derivative, which is introduced and discussed. A result of existence and uniqueness of solutions is stated and proved under the assumption that the transition semigroup associated to the linear part of the equation has a smoothing property, that is, it maps continuous functions into $G$-differentiable ones. The validity of this smoothing assumption is fully discussed for the case of the Ornstein-Uhlenbeck transition semigroup and for the case of invertible diffusion coefficient covering cases not previously addressed by the literature.
It is shown that the results apply to Hamilton-Jacobi-Bellman (HJB) equations associated to
infinite horizon optimal stochastic control problems in infinite dimension
and that, in particular, they cover examples of optimal boundary control of the heat equation that were not treatable with the approaches developed in the literature up to now.
\bigskip
\vskip 0.15cm

\vskip 0.15cm

\bigskip\\
\textbf{Acknowledgements.} The authors are very grateful to Mauro Rosestolato and Andrzej Swiech for valuable discussions and suggestions. In particular, they are indebted to Mauro Rosestolato for Remark \ref{rem:est}.

\end{abstract}

\setcounter{tocdepth}{3}

\tableofcontents

\section{Introduction}
\label{sec:2015-05-01:00}
%

Semilinear elliptic equations with infinitely many variables are an important subject due to their application to time homogeneous stochastic optimal control problems and stochastic  games problems
over an  infinite horizon.
The infinite dimensionality of the variables arises in many applied problems,
e.g., when the dynamics of the state variables is driven by a stochastic delay equation or by a stochastic PDE.
In these cases the resulting Hamilton-Jacobi-Bellman (HJB) equation (associated to the control problem)
or Hamilton-Jacobi-Bellman-Isaacs (HJBI) equation (associated to the game)
are elliptic
equations in infinite dimension.

Only few papers were devoted to study such kind of
elliptic equations in the literature, using mainly three approaches, as follows
{(see the forthcoming book \cite{FGSbook} for a survey of the present literature)}.
\begin{itemize}
  \item The viscosity solution approach, introduced first, for the second order infinite dimensional case, in \cite{Lions88,Lions89sto,Lions89jfa} and then developed
      in \cite{Swiech93,Swiech94,ishii93} and, later, for more specific problems
      in \cite{GozziSwiech00,GozziRouySwiech00,GozziSritharanSwiech05,KelomeSwiech03}, among others. On one hand, this approach allows to cover a big variety of elliptic equations in Hilbert spaces, including  fully nonlinear ones; on the other hand, a regularity theory for viscosity solutions is not available in infinite dimension.
      Viscosity solutions have been employed to treat  elliptic elliptic equations only in few papers; in particular, we mention \cite{GozziRouySwiech00}.
  \item The mild solution approach by means of representation of solutions through backward stochastic differential equations (BSDEs).  In infinite dimension it was introduced in \cite{FuhrmanTessitore02-ann} (for the parabolic case)     and in \cite{FuTe-ell} (for the elliptic case). This method is applicable, so far, only to semilinear equations satisfying a \emph{structural condition}
      on the operators involved and allows to find solutions with a $C^1$-type regularity when the data are accordingly regular. Moreover, it is suitable to solve the associated control problems in the HJB case. The required structural condition, in the HJB case, substantially states that the control can act on the system modifying its dynamics at most along the same directions along which the noise acts. This may be a stringent requirement  preventing the use of this method to solve some important applied problems, e.g. boundary control problems
      (with the exception of the boundary noise case, see \cite{DebusscheFuhrmanTessitore07,Masiero2010}).
       \item The mild solution approach by means of fixed point arguments --- the method used here. This method has been introduced first in \cite{DaPrato85,Havarneanu85} and then developed in
          \cite{CannarsaDaPrato91,CannarsaDaPrato92} and in various other papers (see e.g. \cite{Gozzi95,Gozzi96,GozziRouy96,
Cerrai01,Cerrai01-40,GoldysMaslowski99,Gozzi98,Gozzi02,Masiero03,
Masiero05,Masiero07AMO,Masiero07EJP}\footnote{Similar results, but using a different method based on a convex regularization procedure, were obtained in earlier papers
\cite{BarbuDaPrato81,BarbuDaPrato83book,BarbuDaPrato83siam} in the special case of convex data and quadratic Hamiltonian function $F$.}.
Such method, suitable for semilinear equations, consists in proving first smoothing properties of the transition semigroup associated to the linear part of the equation and then  applying  fixed point theorems. In this way, one finds solutions with $C^1$-type regularity properties, which allow, in some cases, to solve the associated control problems.
Within this approach, the elliptic case has been treated  in the papers
\cite{CannarsaDaPrato92,GozziRouy96,Cerrai01-40,
GoldysMaslowski99,Masiero07AMO}.
\end{itemize}

The main purpose of this paper is the develop a
general framework for the application of the mild solution approach
in the elliptic case and to show that such framework
allows:
\begin{itemize}
  \item on one side, to widely extend the applicability of the mild solution approach by carefully fixing and extending the use of $G$-derivatives
      introduced in \cite{FuTe-gen-gra} and developed in \cite{Masiero03, Masiero05};
  \item on the other side, to cover HJB equations arising in control problems,
  like the boundary control ones, which so far cannot be solved by means of other techniques.
\end{itemize}

We now present the equation we deal with and explain briefly the main
ideas.
We consider the following class
of semilinear elliptic equations in a real separable Hilbert space $H$:
\begin{equation}
\label{eq4:HJbasicelliptic}
\lambda v (x)-\frac{1}{2}\;\mbox{\rm Tr}\;[Q(x)D^2v(x)] -\langle Ax+b(x),Dv(x) \rangle  -F(x,v(x),Dv(x))=0,
\; \; \quad  x \in H.
\end{equation}
Here $\lambda>0$,  the operator $A$ is a linear (possibly unbounded) operator on $H$, and the functions $b: H \to H$, $Q: H \to \call^+(H)$ (where $\call^+(H)$ denotes the set of bounded nonnegative linear operators on $H$), and $F: H \times \R\times H \to H$ are measurable.
Such equations includes HJB equations associated to discounted time homogeneous stochastic optimal control problems in $H$ over infinite time horizon (see Section \ref{sec:SOC}); in this case $F$
is called Hamiltonian. Here, our main focus is  on the application to this latter case.
However, the main results are proved in a more general framework that allows to cover also other cases like HJBI equations associated to differential games.

The type of solutions we study here are called {\em mild solutions}, in the sense
that they solve the equation in
the following integral form:
\begin{equation}
\label{eq4:HJbasicellipticmild}
v (x)= \int_0^{+\infty} e^{-\lambda s}P_{s}[F(\cdot,v(\cdot),Dv(\cdot))]ds,
\; \; \quad  x \in H,
\end{equation}
where $(P_{s})_{s \ge 0}$ is the transition semigroup associated
to the linear part of \myref{eq4:HJbasicelliptic}, that is,  to the operator
$$\cala v =\frac{1}{2}\;\mbox{\rm Tr}\;[Q(x)D^2v] +\langle Ax+b(x),Dv \rangle. $$
Note that, in \eqref{eq4:HJbasicellipticmild} only the gradient of $v$ appears.
However, as it usually arises in applications to control problems, the dependence
of the nonlinear part $F$ on the gradient $Dv$ can occur in a special form, that is, through a family of linear, possibly unbounded, operators $G$; this leads to consider a
generalized concept of gradient, which we call $G$-gradient and denote by $D^G$ (see Subsection \ref{SS4:GDER} for details). Hence, we  actually have a nonlinear term in the form
$F(x,v(x),Dv(x))=F_0(x,v(x),D^{G}v(x))$, where $F_0$ is a suitable function.
We prove existence and
uniqueness of solutions to \eqref{eq4:HJbasicellipticmild}, with $F(x,v(x),Dv(x))$ replaced by $F_0(x,v(x),D^{G}v(x))$, by applying a fixed point argument  and
assuming, to let the method work, a suitable smoothing property of the transition semigroup $P_s$.
The good news are that, unlike viscosity solutions, this method provides, by construction, a solution that enjoys the minimal regularity needed to define, in classical sense, the candidate optimal feedback map of the associated stochastic optimal control problem and, unlike the BSDEs approach,
the structural restriction ${\rm Im} \,G \subseteq {\rm Im}\, Q^{1/2}$ is here not required.
On the other hand, the required smoothing property is not trivial to prove and
fails to hold in many cases. However, we show that, using the generalization of $G$-derivative that we introduce here, it is is satisfied in some important classes of
problems --- where instead the structural condition ${\rm Im}\, G \subseteq {\rm Im}\, Q^{1/2}$ is not true.

 Mild solutions are not regular enough to apply of It\^o's formula yet --- hence, to enable to prove a verification theorem showing that the candidate optimal feedback map really provides a solution to the associated control problem. Nevertheless, they represent a first step towards this goal.
Indeed, one can rely on this notion to prove that they are, in fact, \emph{strong solutions} (see \cite{Gozzi95,GozziRouy96}): the latter concept allows to perform the verification issue by approximation. {On the other hand,
at least in the case of control of Ornstein-Uhlenbeck process, already the notion of mild solution suffices to prove such a kind of result, as we will show in a subsequent paper
(see also Remark \ref{rem:OFM} on this issue).}
\medskip

This paper is organized as follows. In Section \ref{sec:pre}, after the setting of notations and spaces, we introduce and study the notion of $G$-derivative for functions between  Banach spaces. This is a kind of generalized Gateaux differential, where only some directions, selected by an operator valued map $G$, are involved. The latter notion was considered and studied in some previous papers. Precisely, it was developed in \cite{FuTe-gen-gra} (see also \cite[Sec.\,4]{Masiero03}
and \cite{Masiero05}) for maps $G$ valued in the space of bounded linear operators. Here we extend  this notion, fixing some features, to the case when the map $G$ is valued in the space of \emph{possibly unbounded} linear operators \footnote{This extension enables to treat a larger variety of cases (in particular the case of boundary  control problems), which require to deal with unbounded control operators when the problem is reformulated in infinite dimension.}. The crucial property that we prove is represented by a ``pointwise" exchange property between $G$-derivative and integration (Proposition \ref{intder}), on which our main result relies.

Section \ref{sec:sem} is the theoretical core of the paper. We set the notion of mild solution motivating it by an informal argument and state our main results  (Theorems \ref{th4:exunHJBgeneralell} and \ref{Th:strongfeller}) on existence and uniqueness of solutions to   the integral equation
\begin{equation}
\label{eq4:HJbasicellipticmildF0}
v (x)= \int_0^{+\infty} P_{s}[F_0(\cdot,v(\cdot),D^Gv(\cdot))]ds,
\; \; \quad  x \in H.
\end{equation}
The results are stated under the aforementioned smoothing assumption: we require that the semigroup $P_s$ maps continuous functions into $G$-differentiable ones.

To show that the smoothing assumption is actually verified in several concrete circumstances, we devote Section \ref{S4:SMOOTHING} to the investigation of reasonable conditions guaranteeing the validity of it. In particular, we focus on the Ornstein-Uhlenbeck case, providing a new result that falls in the previous literature when $G=I$, but extend meaningfully to other important cases when $G\neq I$ and, especially, when $G$ is unbounded. The result is contained in Theorem \ref{th4:regOUG} and extends the known one Theorem \ref{th4:regOU} (contained in \cite{DaPratoZabczyk02}). For completeness, we also report another known result (Theorem \ref{th4:FTBismut}) contained in \cite{FuhrmanTessitore02-sto}, where the smoothing assumption is verified for $G=I$ in the case of smooth data and invertible diffusion coefficient.

Finally, to show the implications of our results, we devote Section \ref{sec:SOC}
to present a stochastic optimal control in the Hilbert space $H$ and show how the associated Hamilton-Jacobi-Bellman (HJB) equation falls, as a special case, in the class of \eqref{eq4:HJbasicelliptic}. Then, a specific example of boundary optimal control, through Neumann type conditions, of a stochastic heat equation is provided. In this example, we discuss the validity of all the assumptions that allow to apply our main result through the use of Corollary \ref{corcor}.  As far as we know, this is the first time that the HJB equation associated to this kind of problem is approached by means of solutions that have more regularity than viscosity solutions.

\section{Preliminaries}\label{sec:pre}
In this section we provide some preliminaries about spaces and notations used in the rest of the paper. Also, we provide the notion of $G$-gradient for functions defined on Banach spaces and some properties of this object.
\subsection{Spaces and notation}
Here we introduce some spaces and notations.
\subsubsection{General notation and terminology}
If $U$ is a Banach space we denote its
norm by $|\cdot|_U$. The weak topology on $U$ is denoted by $\tau^U_w$.
If $U$ is also Hilbert, we denote its inner product
by $\left \langle \cdot,\cdot\right\rangle_U$.
Given $R>0$ and $x_0\in U$, the symbol $B_{U}(x_0,R)$ denotes the closed ball in $U$ centered at $x_0$ of radius $R$.  In all the notations above, we  omit the subscript  if the context is clear.

If a sequence $(x_n)_{n\in\mathbb{N}}\subseteq U$, where $U$ is Banach,  converges to $x\in U$ in the norm (strong) topology we write $x_n\to x$.
 If it converges in the weak topology  we write $x_n\rightharpoonup x$.

If $U$ is a Banach space, we denote by $U^*$ its topological dual, i.e. the space of all continuous
linear functionals defined on $U$.
The operator norm in $U^*$ is denoted by $|\cdot|_{U^*}$.
 The duality with $U$ is denoted by $\< \cdot, \cdot \>_{\<U^*, U\>}$.
 If $U$ is a Hilbert space, unless stated explicitly, we always identify its dual $U^*$ with $U$ through the standard Riesz identification.

%

All the topological spaces are intended endowed with their Borel $\sigma$-algebra. By {measurable} set (function), we always intend {\emph{Borel}} measurable set (function).

\subsubsection{Spaces of linear operators}
If $U,V$ are Banach spaces  with norm $|\cdot|_U$ and $|\cdot|_V$, we denote by
$\mathcal{L}(U,V)$ the set of all bounded (continuous) linear operators
$T :  U \to V$ with norm $|T|_{\mathcal{L}(U,V)}:= \sup_{x \in U, x \ne 0} \frac{|Tx|_V}{|x|_U}$, using for simplicity the notation $\mathcal{L}(U)$ when $U=V$. $\mathcal{L}(U)$ is a Banach algebra with identity element $I_U$ (simply $I$ if unambiguous).

If $U,V$ are Banach spaces, we denote by $\call_u (U,V)$ the space of
 closed densely defined possibly unbounded linear operators
 $T:\mathcal{D}(T)\subseteq U \to V$, where $\cald(T)$ denotes the domain. Clearly, $\call (U,V)\subseteq \call_u(U,V)$. Given  $T\in \call_u (U,V)$,  we  denote its adjoint operator by $T^*:\mathcal{D}(T^*)\subseteq V^* \to U^*$ and its range by $\calr(T)$.

  Let $U$ be a separable Hilbert space. We denote by   $\call_1(U)$ (subset of $\call(U)$) the set of trace class operators, i.e. the operators $T\in\call(U)$ such that, given an orthonormal basis  $\{e_k\}_{k\in\N}$  of $U$, the quantity
 $$|T|_{\call_1({{U}})}\coloneqq\sum_{k=1}^\infty \langle (T^*T)^{1/2} e_k,e_k\rangle_U$$
 is finite.  The latter quantity is independent of the basis chosen and defines a norm making $\call_1(U)$ a Banach space.
  The trace of an operator $T\in\call_1(U)$ is denoted by $\mbox{Tr}[T]$, i.e. $\mbox{Tr}[T]\coloneqq \sum_{k=0}^\infty \langle T e_k,e_k\rangle_U$. The latter quantity is
 is finite and, again,  independent of the basis chosen.
 We denote  by  $\call_1^+(U)$ the subset of $\call_1(U)$ of self-adjoint nonnegative  (trace class) operators on $U$. Note that, if $T\in\call_1^+(U)$, then $\mbox{Tr}[T]=|T|_{\call_1(U)}$.

 If $U,V$ are separable Hilbert spaces, we denote by $\call_2(U,V)$ (subset of $\call(U{{,V}})$) the space of Hilbert-Schmidt operators from $U$ to $V$, that is the spaces of operators such that, given an orthonormal basis  $\{e_k\}_{k\in\N}$  of $U$, the quantity
 $$|T|_{\call_2({{U,V}})}\coloneqq\left(\sum_{k=0}^\infty |Te_k|_V^2\right)^{1/2}$$
 is finite.  The latter quantity is independent of the basis chosen and defines a norm making ${{\call_2(U,V)}}$ a Banach space.  It is actually a Hilbert space with the scalar product
$$
\langle T, S\rangle_{\call_2(U,V)}\coloneqq \sum_{k=0}^\infty \langle Te_k, Se_k\rangle_V,$$
where $\{e_k\}_{k\in\N}$ is any orthonormal basis of  $U$.
We refer to \cite[App.\,A]{DaPratoZabczyk14} for more details on trace class and Hilbert-Schmidt operators.

%

%
%

\subsubsection{Function spaces}
Let $U, V,Z$ be  Banach spaces and $m\geq 0$.
We denote by $B(U,V)$ (respectively,  $B_b(U,V)$) the space of measurable (respectively, measurable and bounded) functions from $U$ into $V$.
The space $B_b({U},V)$ is a Banach space
with the usual norm
\begin{equation}\label{supnorm}
|\varphi |_{0} =  \sup_{x \in {U}}|\varphi (x)|_{V}.
\end{equation}

We denote by  $C({U},V)$ (respectively,
$C_b({U},V)$)  the space of  continuous (respectively,
continuous and bounded) functions from $U$ into $V$.
The space $C_b({U},V)$ is a Banach space
with the norm \eqref{supnorm}.

Given $m\geq 0$,
we define  $B_m(U,V)$ (respectively, $C_m({U},V)$) as the set of all functions $\phi \in {B}({U},V)$ (respectively, $\phi\in C(U,V)$) such that the function
\begin{equation}\label{eqappA:psiphipolgrowth}
\psi(x):= \frac{\phi(x)}{{1+|x|^m}}, \ \ \ x\in U,
\end{equation}
belongs to $B_b(U,V)$ (respectively, ${C}_b({U},V)$). These spaces are made by
functions that have at most polynomial growth of order $m$ and are Banach spaces when they are endowed with the norm
$|\phi|_{B_m({U},V)} :=\sup_{x\in {U}} \frac{|\phi(x)|}{{1+|x|^m}}$
(respectively,
$\ |\phi|_{C_m({U},V)}:=\sup_{x\in {U}} \frac{|\phi(x)|}{{1+|x|^m}}).$
We will write $|\phi|_{B_m}$ (respectively, $|\phi|_{C_m}$) when the spaces are clear from the context.
When $m=0$ the above spaces
reduces to $B_b(U,V)$ and $C_b({U},V)$ and we keep this notation to refer to them.

If  $f:{U} \to V$, the Gateaux (resp., Fr\'echet) derivative of $f$ at the point $x$ is denoted by $\nabla f(x)$ (resp., $D f(x)$).


%
%
%
%
%

We define the space $C^{s}({U},\call(Z,V))$
as the space of maps $f:{U}\rightarrow \call\left(Z,V\right)$ such that, for every $z\in Z$,
$f(\cdot)  z\in C\left({U},V\right)$(\footnote{This property is usually called  strong continuity of $f$.}) and
 the space $C_{m}^{s}({U},\call(Z,V))$
as the space of maps $f:{U}\rightarrow \call\left(Z,V\right)$ such that, for every $z\in Z$,
$f(\cdot)  z\in C_{m}\left({U},V\right)$.
When $m=0$, we write  $C_{b}^{s}\left({U},\call\left(  Z,V\right)  \right)$.
\begin{Proposition}
Let $m\geq 0$ and let $U,V,Z$ be three Banach spaces. The space
 $C_{m}^{s}\left(  {U},\call\left(  Z,V\right)\right)$) is Banach
 when endowed with the norm
\begin{equation}\label{normstrong}
\left| f\right|_{C _{m}^{s}\left({U},\call\left(Z,V\right)\right)}
:=\sup_{x\in {U}}\frac{\left| f\left(  x\right)  \right| _{\call\left(
Z,V\right)}}{{1+|x|^m}}.
\end{equation}
(When it is clear from the context, we simply write
 $\left| f\right|_{C_{m}^{s}}$.)
\end{Proposition}
\begin{proof}
First of all, we observe that the right hand side of \eqref{normstrong} is finite due to a straightforward application of the Banach-Steinhaus theorem, so it clearly defines a norm.

Let  $(f_n)_{n\in\N}$ be a Cauchy sequence in ${C _{m}^{s}\left({U},\call\left(Z,V\right)\right)}$.
Then, for each $x\in U$, by completeness of $\call(Z,V)$, $f_n(x)\to f(x)$ in $\call(Z,V)$ for some $f(x)\in \call(Z,V)$. On the other hand, by completeness of $C_m(U,V)$, we also have, for each $z\in Z$, where $f_n(\cdot)z\to f_z(\cdot)$ in $C_m(U,V)$ for some $f_z\in  C_m(U,V)$. By uniqueness of the limit we have $f(x)z=f_z(x)$ for each $z\in Z$ and $x\in U$.  Hence, $f\in C _{m}^{s}\left({U},\call\left(Z,V\right)\right)$. It remains to show that $f_n\to f$ with respect to $\left| f\right|_{C_{m}^{s}\left({U},\call\left(Z,V\right)\right)}$, that is
$$
\sup_{x\in U}\sup_{|z|_Z=1} \frac{|(f_n(x)-f(x))z|}{1+|x|^m}\to 0.
$$
Now, for every $z\in Z$ with $|z|_Z=1$ and $n\in N$, we have
 $$\frac{|(f_n(x)-f(x))z|}{1+|x|^m}=\lim_{k\to \infty} \frac{|(f_n(x)-f_k(x))z|}{1+|x|^m}\leq \limsup_{k\to \infty} |f_n-f_k|_{C_{m}^{s}\left({U},\call\left(Z,V\right)\right)}, \ \ \ \forall x\in U.$$
 We conclude as   $(f_n)_{n\in\N}$ is a Cauchy sequence in ${C_{m}^{s}\left({U},\call\left(Z,V\right)\right)}$.
\end{proof}
\subsubsection{Spaces of stochastic processes}
Let $(\Omega, \mathcal{F}, (\mathcal{F}_t)_{t\geq 0}, \P)$ be a filtered probability space satisfying the usual conditions.
Given $p\geq 1$, $T>0$, and a Banach space $U$, we denote by $\mathcal{H}_\calp^{p,T}(U)$ the set of all (equivalence classes of) progressively measurable processes
$X\colon [0,T] \times \Omega\to U$ such that
\[
|X(\cdot)|_{\mathcal{H}_\calp^{p,T}(U)} := \left ( \sup_{s\in [0,T]} \mathbb{E} |X(s)|_U^p \right )^{1/p} < +\infty.
\]
This is a Banach space with the norm $|\cdot|_{\mathcal{H}_\calp^{p,T}(U)}$.
Next, we denote by $\mathcal{H}_\calp^{p,loc}(U)$ the space of all (equivalences classes of) progressively measurable processes
$X\colon [0,+\infty) \times \Omega\to E$ such that
$X|_{[0,T]\times\Omega}\in \mathcal{H}_\calp^{p,T}(U)
$
for every $T>0$.

Given $p\geq 1$, $T>0$, and a Banach space $U$, we denote by
$\mathcal{K}^{p,T}_{\mathcal{P}}(U)$ the space
 of all (equivalence classes of) progressively measurable processes
$X\colon [0,T] \times \Omega\to U$ admitting a version with continuous trajectories and  such that
\[
|X(\cdot)|_{\mathcal{K}_\calp^{p,T}(U)} := \E \left[\sup_{s\in [0,T]} |X(s)|_U^p \right ]^{1/p} < +\infty.
\]
This is a Banach space with the norm $|\cdot|_{\mathcal{K}_\calp^{p,T}(U)}$.
We denote by $\mathcal{K}_\calp^{p,loc}(U)$ the set of all (equivalences classes of) progressively measurable processes
$X\colon [0,+\infty) \times \Omega\to U$ such that
$X|_{[0,T]\times\Omega}\in \mathcal{K}_\calp^{p,T}(U)
$
for every $T>0$.

\subsection{Bochner integration}
Let $I\subseteq\R$, let $V$ be a Banach space, and let $f:I\to V$ be measurable.
We recall that, if $|f|_V\in L^1(I,\R)$, then $f$ is Bochner integrable, and we write $f\in L^1(I,V)$. Moreover, in this case
\begin{equation}\label{ppppk}
\bigg\langle v^*,\  \int_I f(t)dt\bigg\rangle_{V^*,V} =\ \int_I\langle v^*,f(t)\rangle_{V^*,V}dt, \ \ \ \forall v^*\in V.
\end{equation}
Finally,  we recall that, if $V$ is separable, by Pettis measurability Theorem \cite[Th.\,1.1]{Pettis38}, $f$ is measurable if and only if  $t\mapsto \langle v^*, f(t)\rangle_{V^*,V}$ is measurable for every $v^*\in V^*$.

\subsection{$G$-derivative}
\label{SS4:GDER}
Here we set and investigate the notion of
$G$-derivative for functions $f:{U}\rightarrow V$, where  $U,V$ are Banach spaces. The latter notion was defined in \cite{FuTe-gen-gra} (see also \cite[Sec.\,4]{Masiero03}
and \cite{Masiero05}) when $G$ is a map $G:{U}\to \call(Z,U)$ with $Z$ Banach space. Here we extend the definition requiring only that $G:{U}\to \call_u(Z,U)$.

\begin{Definition}
\label{df4:Gderunbounded}
Let $U$, $V$, $Z$ be three Banach spaces and let $f:{U} \rightarrow V$,
 $G: U\to \call_u (Z,U)$.
\begin{itemize}
\item[(i)]
The $G$-directional derivative $\nabla^{G}f\left(x;z\right)$
at a point $x\in {U} $ along  the direction
$z\in \cald(G(x))\subseteq Z$ is defined as
\begin{equation}
\nabla^{G}f\left(  x;z\right) : =\lim_{s\rightarrow 0}\frac{f\left(  x+sG\left(
x\right)  z\right)  -f\left(  x\right)  }{s}, \ \ \ \ s\in\mathbb{R},
\label{eq4:Gder}
\end{equation}
 assuming that this limit exists in $V$.

  \item[(ii)] We say that $f$ is $G$-Gateaux differentiable
at a point $x\in {U}$ if it admits the $G$-directional derivative along every
direction $z\in \cald(G(x))$ and there exists a {\em bounded linear operator}  $\nabla^{G}f\left(  x\right)
\in \call\left(  Z,V\right)$ such that
$\nabla^{G}f\left(  x;z\right)  =\nabla^{G}f\left(  x\right)z$
for every $x\in U$ and every $z \in \cald(G(x))$.
In this case $\nabla^{G}f\left(  x\right)$ is
called the $G$-Gateaux derivative at $x$.

\item[(iii)] We say that $f$ is $G$-Gateaux
differentiable on ${U}$ if it is $G$-Gateaux differentiable at every point $x\in
{U}$ and call the object $\nabla^Gf:{U}\to \call(Z,V)$ the   $G$-Gateaux derivative of $f$.
  \item[(iv)] We say that $f$ is $G$-Fr\'echet differentiable
(or simply $G$-differentiable)
at a point $x\in {U}$ if it is $G$-Gateaux differentiable at $x$ and if the limit
in (\ref{eq4:Gder}) is uniform for $z\in B_Z(0,1)\cap \cald (G(x))$. The latter is equivalent to
$$
\lim_{z\in\cald(G(x)),\ |z|_U\to 0} \frac{|f(x+G(x)z)-f(x)-\nabla^Gf(x){{z}}|_V}{|z|_U}=0.
$$
 In this case,
we denote $D^G f(x):= \nabla^Gf(x)$ and call $D^Gf{{(x)}}$ the $G$-Fr\'echet derivative (or simply the $G$-derivative) of $f$ at $x$.
\item[(v)]
We say that $f$ is $G$-Fr\'echet differentiable on $U$ (or simply $G$-differentiable) if it is $G$-Fr\'echet differentiable at every point $x \in {U}$ and call the object $D^Gf:{U}\to \call(Z,V)$ the   $G$-Fr\'echet derivative  (or simply the $G$-derivative) of $f$.
\end{itemize}
\end{Definition}
Note that, in the definition of the $G$-derivative, one considers only the directions
in $U$ selected by the image of $G(x)$.
This is similar to what is often done in
the theory of abstract Wiener spaces, considering the $K$-derivative, where $K$ is a subspace of $U$,
see e.g. \cite{Gross67}.
Similar concepts are also used in \cite{DaPrato85}, \cite{Priola99} and in
 \cite[Sec.3.3.1]{DaPratoZabczyk02}.
A generalized notion of  $G$-derivative in spaces $L^p(H, \mu)$ where $H$ separable Hilbert space and $\mu$ is a suitable Radon measure,
is considered in relation to Dirichlet forms, see e.g. \cite{MaRockner92}
(or also \cite[Ch.\,3]{DaPrato14}, where it is called Malliavin derivative). In all these references $G$ is a bounded operator.

\medskip

Clearly, if $f$ is Gateaux (resp., Fr\'echet) differentiable  at $x\in {U}$, and  $G:{U}\to \call (Z,U)$, it turns out that $f$ is $G$-Gateaux (resp., Fr\'echet) differentiable at $x$
and
\begin{equation}\label{eq4:DGDfG}
    \nabla^{G}f\left(x\right)z  =\nabla f\left(  x\right)(G\left(  x\right)z)
\qquad \left(\mbox{resp., } D^{G}f\left(x\right)z  =D f\left(  x\right)(G\left(  x\right)z)
 \right),
\end{equation}
i.e. the $G$-directional derivative in direction $h\in Z$ is just the usual
directional derivative at a point $x\in U$ in the direction
$G\left(  x\right)  z\in U$.

If $V=\R$, then the (G\^ateaux or Fr\'echet)   $G$-derivative takes values in
$\call(Z,\R)=Z^*$. Hence,
if $Z$ is a Hilbert space,  identifying $Z$
with its topological dual $Z^*$,  we consider $\nabla^Gf:{U}\to Z$ and  write
$\<\nabla^G f(  x), z\>_Z$ for $\nabla^{G}f\left(x\right)z$. Similarly we do for the $G$-Fr\'echet
derivative.

In the same spirit, when $V=\R$ and both $U$ and $Z$ are Hilbert spaces,
we identify the spaces $U$ and $Z$ with their topological  dual spaces $U^*$ and $Z^*$. Hence, whenever $f$
is Gateaux (resp., Fr\'echet) differentiable at $x\in {U}$, the identity
\myref{eq4:DGDfG} becomes
$$
\<\nabla^G f(  x), z\>_Z  =\< \nabla f(  x),G\left(  x\right)  z \>_U
=\<G\left(  x\right)^*\nabla f(  x), z\>_Z,
$$
and similarly for the $G$-Fr\'echet derivatives.

\begin{Example}\label{ex4:Gnondiff}
 The notion of  $G$-derivative
allows  to deal with functions which are not Gateaux differentiable, as shown by the following example.

Let $f:\R^{2}\rightarrow \R$ be defined by $f(x_1,x_2)\coloneqq\left\vert
x_{1}\right\vert x_{2}$. Clearly, $f$ does not admit   directional derivative in the direction $(1,0)$ at the point $(x_1,x_2)=(0,1)$. On the other hand,
if we consider $G:\R^2\to \call(\R,\R^2)\cong \R^2$, defined by $G(\cdot)\equiv G_0$, where  $G_0=(0,1)$, then
 $f$ admits
$G$-Fr\'echet derivative at every $(x_1,x_2) \in \R^2$.
\end{Example}

\begin{Remark}
When $G$ only takes values in $\call_u(U,Z)$, that is when we deal with the possibility that $G(x)$ is unbounded, then
even if $f$ is Fr\'echet differentiable at all points $x\in {U}$,
the $G$-Fr\'echet derivative may not exist in some points.
Indeed, consider the following example.

Let $U, Z$ be  Hilbert spaces, let
$G_0:\cald(G_0)\subseteq Z \to U$ be a closed densely defined unbounded linear operator on $U$, and let $G_0^*:U\to Z$ be its (unbounded) adjoint. Next, let $G:U\to \call_u(Z,U)$ defined by $G(\cdot)\equiv G_0$ and
let $f:U \to \R$ be defined by $f(x)\coloneqq |x|^2$. Clearly, $f$ is Fr\'echet differentiable  at every $x \in U$ and $Df(x)= 2x$. Then,
by definition of $G$-directional derivative, we have
$$\nabla^{G}f\left(  x;z\right)  =\<Df\left(  x\right),G_0z\>_U=2\langle x,G_0z\rangle_U, \ \ \ \forall x\in U, \ \forall z\in \cald (G_0).
$$
On the other hand, if $f$ was $G$-Fr\'echet differentiable at every $x\in U$, we should have $D^G f(x) \in U$ for every $x\in U$,  and therefore
$$
|\nabla^{G}f\left(  x;z\right)|=\left|\<D^G f(x),z\>_Z \right|\le \left|D^G f(x)\right|_Z|z|_Z,\quad \forall z\in \cald(G_0).
$$
It should follow $\cald(G_0^*)=U$, which is not the case if $G_0$ is any genuinely unbounded linear operator.
\end{Remark}

We now define, following \cite{FuhrmanTessitore02-ann}
and \cite{Masiero03,Masiero05}, some relevant classes of spaces of $G$-regular functions.

\begin{Definition}
\label{df4:Gspaces}
Let $U$, $V$, $Z$ be  Banach spaces, let $G:{U}\to\call_u(Z,U)$, and
let $m\ge 0$.
We define the spaces of functions
\begin{equation}
\mathcal{G}_{m}^{1,G}\left(  {U},V\right)  :=\left\{
f\in C_{m}\left({U},V\right):\,\, f \mbox{is} \ G\mbox{-Gateaux differentiable on} \ U \ \mbox{and} \ \nabla^{G}f \in
{C}_{m}^s\left({U},\call(Z,V)\right)\right\}, \nonumber
\label{GGclass limitata}
\end{equation}
\vspace{-15pt}
\begin{equation}
{C}_{m}^{1,G}\left(  U,V\right):=\left\{
f\in C_{m}\left(  U,V\right):\, f \mbox{is} \ G\mbox{-{{Fr\'echet}} differentiable on} \ U \ \mbox{and} \ D^{G}f
\in C_{m}\left(U,\call(Z,V)\right)\right\}\nonumber
\label{CGclass limitata}
\end{equation}
When $m=0$ we use the notation
${\mathcal{G}}_{b}^{1,G}\left({U},V\right)$ and
${C}_{b}^{1,G}\left(  {U},V\right)$. Moreover, when $V=\R$ we omit it in the notation.

\end{Definition}

\begin{Definition}[Pseudoinverse]
\label{df:pseudoinverseapp}
Let $Z$ be a uniformly convex Banach space, let $U$ be a Banach space, and let $T\in \mathcal{L}_u(Z,U)$. The \emph{pseudoinverse} $T^{-1}$ of $T$ is the linear operator defined on $T(Z) \subseteq U$ and valued in $\cald(T)\subseteq Z$ associating to each $x\in \calr(T)$ the (unique) element in $T^{-1}({{\{x\}}})$ with minimum norm  in $Z$. {{\footnote{Existence and uniqueness of such an element follows from the fact that $T$ is a closed operator and applying the results of \cite[Sec.\,II.4.29, p.\,74]{DunfordSchwartz58-I}).}}}
\end{Definition}
\begin{Remark}\label{rem:inverse}
If $Z$ in Definition \ref{df:pseudoinverseapp} is Hilbert space,  then $
\calr(T^{-1})= (\ker T)^\perp.
$
\end{Remark}
Now we deal with the possibility of performing the $G$-differentiation under the integral sign in pointwise and functional sense. Due to the integrability issues clarified at the beginning of Section \ref{sec:sem}, we will make use only of the pointwise exchange property (Proposition \ref{intder}). However, since the analogue functional property has not been well developed in the literature, we establish the result also in this case  (Corollary \ref{cr4:DGclosed}) providing a complete proof.

\begin{Assumption}\label{assG}
$U$ and $Z$ are Hilbert spaces.
The map $G:{U}\to \call_u(Z,U)$ is such that
 \begin{enumerate}[(i)]
 \item $\cald(G(x))=\cald (G(y))$ for every $x,y\in U$; we denote by $\cald_G$ the common domain;
  \item $\calr(G(x))=\calr(G(y))$ for every $x,y\in U$; we denote by $\calr_G$ the common range;
 \item\label{iip3} Let $G^{-1}(x)$ be the pseudo-inverse of $G(x)$ according to Definition \ref{df:pseudoinverseapp}.
 The map $x\mapsto G(x)^{-1}y$ is locally bounded for every $y\in\calr_G$.
 \end{enumerate}
\end{Assumption}
\begin{Proposition}\label{intder}
Let Assumption \ref{assG} hold. Let $m\geq 0$  and let $f:[0,+\infty)\times U\to V$ be  measurable and such that
\begin{enumerate}[(i)]
\item
 $f(t,\cdot)\in \calg_m^{1,G}(U,V)$ (resp., $C_m^{1,G}(U,V)$) for a.e. $t\in[0,+\infty)$;
 \item there exists $g\in L^1([0,+\infty),\R)$ such that, for a.e. $t\in[0,+\infty)$ and every $x\in U$,
 \begin{equation}\label{stima2}
|f(t,x)|_{V}\leq g(t)(1+|x|^m),
\end{equation}
\vspace{-20pt}
\begin{equation}\label{stima1}
|\nabla^Gf(t,x)|_{\call(Z,V)}\leq g(t)(1+|x|^m) \ \ \mbox{(resp., } \ |D^Gf(t,x)|_{\call(Z,V)}\leq g(t)(1+|x|^m)).
\end{equation}
\end{enumerate}
Then,  the function $L:U\to V$,  $L(x)\coloneqq\int_0^{+\infty}f(t,x)\,dt$ is well defined and belongs to $\calg_m^{1,G}(U,V)$ (resp., $C_m^{1,G}(U,V)$) and
\begin{equation}\label{scambio}
\nabla^G L(x)h =  \int_0^{+\infty} \nabla^Gf(t,x)h\,dt, \ \ \ \mbox{(resp., } D^G L(x)h =  \int_0^{+\infty} D^Gf(t,x)h\,dt), \ \ \ \forall x\in U, \ \forall h\in \cald_G.
\end{equation}
\end{Proposition}
\begin{proof}
First of all, note that the fact that $L$ is well defined and belongs to $C_m(U,V)$ follows from (i) and \eqref{stima2} by dominated convergence.

We prove now the other claims for the $G$-Gateaux gradient; the proof for the $G$-Fr\'echet gradient is obtained just replacing $\nabla^G$ by $D^G$ and $C_m^s(U,\call(Z,V))$ by $C_m(U,\call(Z,V))$ in the following.
Let $x\in U$, $h\in \cald_G$. Set
$$
  y(r)\coloneqq x+rG(x)h, \ \ \  \varphi(t,r)\coloneqq f(t,y(r)), \ \  t\geq 0, \  r\in\R.
$$
Then, using Assumption \ref{assG}(i)-(ii), we can write
\begin{align*}
\frac{\partial}{\partial r}\varphi(t,r)&=\lim_{\xi\to 0}\frac{f(t,y(r)+\xi G(x)h)-f(t,y(r))}{\xi}\\
&=\lim_{\xi\to 0}  \frac{f(t,y(r)+\xi G(y(r))G(y(r))^{-1}G(x)h)-f(t,y(r))}{\xi}\\
&=\nabla^G f(t,y(r))(G(y(r))^{-1}G(x)h).
\end{align*}
By Assumption \ref{assG}(\ref{iip3}) and \eqref{stima1}, choosing $\varepsilon>0$ sufficiently small, we get
\begin{equation}
\left|\frac{\partial}{\partial r} \varphi(t,r)\right|\leq g_x(t), \ \ \ \forall r\in(-\varepsilon,\varepsilon).
\end{equation}
Then,  arguing as in the standard one dimensional  case, dominated convergence combined with a generalization of Lagrange Theorem to $V$-valued functions (see \cite[Prop.\,3.5,\,p.\,76]{Zeilder86-I}) yields
$$\nabla^G L(x;h) =  \int_0^{+\infty} \nabla^Gf(t,x)h\,dt.$$
This also shows that
$$|\nabla^G L(x;h)|_V\leq \left(\int_0^{+\infty} |\nabla^Gf(t,x)|_{\call(Z,V)}\,dt\right) |h|\leq \left(\int_0^{+\infty} |g_x(t)|\,dt\right) |h|, \ \ \ \forall x\in H, \forall h\in \cald_G.$$
Hence, as $g_x\in L^1([0,+\infty),\R)$,
we conclude that $L$ is $G$-Gateaux differentiable on $U$ and \eqref{scambio} holds true.
 The fact that $\nabla^GL\in C_m^s(U,\call (Z,V))$, hence $L\in \calg^{1,G}_m(U,V)$,  comes from \eqref{scambio} by dominated convergence.
\end{proof}
\begin{Assumption}\label{assGbis}  Assumption \ref{assG} holds true with \emph{locally bounded} replaced by \emph{continuous} in (iii).
 \end{Assumption}
\begin{Proposition}
\label{lm4:GBanach}
Let Assumption \ref{assGbis} hold.
The spaces of Definition \ref {df4:Gspaces} are Banach when
endowed with the norm
\begin{equation}
\left| f\right| _{{\calg}_{m}^{1,G}({U},V)}:=
\sup_{x\in {U}}\frac{\left| f(x)\right|_{V}}
{{1+|x|^m}}
+\sup_{x\in {U}}\frac{\left| \nabla^{G}f(x)\right| _{\call\left(  Z,V\right)  }}
{{1+|x|^m}}. \label{gnorm}
\end{equation}
(For elements of
$C_{m}^{1,G}\left({U},V\right)$
 we use the notation $\left| f\right|_{C_{m}^{1,G}\left({U},V\right)}$. If clear from the context, we simply write
$\left| f\right|_{\calg_{m}^{1,G}}$ and $\left| f\right| _{C_{m}^{1,G}}$.)
\end{Proposition}

\begin{proof}
We give the proof for $\calg_m^{1,G}$. The proof
$C_m^{1,G}$ is analogous.

Let $\left(\Phi_{n}\right)_{{n\in\mathbb{N}}}$ be a Cauchy sequence in
$\calg_{m}^{1,G}\left(  {U},V\right)$. In particular,
$\left\{\Phi_{n}\right\}_{{n\in\mathbb{N}}}$ is a Cauchy sequence in
$C_{m}\left({U},V\right)$, so that $\Phi_{n}$ converges to a
function $\Phi\in C_m(U,V)$.

Now, for all $x \in U$,
$\left(\nabla^{G}\Phi_{n}\left(  x\right)\right)_{{n\in\mathbb{N}}}$ is a Cauchy sequence of linear bounded operators in $\call\left(
Z,V\right)$, so that $\nabla^{G}\Phi_{n}\left(  x\right)  $ converges to a
linear bounded operator $A(x)$.
On the other hand, for all $z \in Z$, the sequence
$\left(\nabla^{G}\Phi_{n}(\cdot)z\right)_{{n\in\mathbb{N}}}$ is a Cauchy sequence
in $C_{m}\left({U},V\right)$ so that $\nabla^{G}\Phi_{n}(\cdot)z$ converges to a
function $A_z \in C_{m}\left({U},V\right)$.
Hence, we have $A_z(x)=A(x)z$, which yields $A\in C_m^s(U,\call(Z,V))$.

Now, note that,  by definition of $\nabla^G\Phi_n(x)$, we have $\nabla^G\Phi_n(x)h=0$ whenever $h\in\mbox{ker}(G(x))$. It follows
\begin{equation}\label{kerker}
\mbox{ker}(A(x))\supseteq \mbox{ker}(G(x)), \ \ \forall x\in U.
\end{equation}

We are going prove   that
$\Phi\in \calg_{m}^{1,G}\left(  {U},V\right)$
and $A=\nabla^{G}\Phi$.
Let $x\in U$ and $h\in \cald_G$. Set, for $r\in[-1,1]$,  $y(r)\coloneqq x+rG(x)h$ and
$
\varphi_n(r)\coloneqq \Phi_n(x+rG(x)h).
$
Then,
arguing as in the proof of Proposition \ref{intder}, we get
\begin{eqnarray*}
\varphi'_n(r)=
\nabla^G\Phi_n(y(r))G(y(r))^{-1}G(x)h, \ \  \ r\in[-1,1].
\end{eqnarray*}
As $\nabla^G\Phi_n\in C_m^s(U,\call(Z,V))$, by the Banach-Steinhaus Theorem the family $\{\nabla^G\Phi_n(y(r))\}_{r\in[-1,1]}$ is a family of uniformly bounded operators in $\call(U,V)$. Setting
$$M\coloneqq  \sup_{\alpha\in[-1,1]}|\nabla^G\Phi_n(y(\alpha))|_{\call(Z,V)}$$
  we have, for every $r,s\in[-1,1]$,
\begin{eqnarray*}
|\varphi'(r)-\varphi'(s)|&=&|(\nabla^G\Phi_n(y(r))G(y(r))^{-1}-\nabla^G\Phi_n(y(s))G(y(s))^{-1})G(x)h|_V\\
&&\leq |\nabla^G\Phi_n(y(r))|_{\call(Z,V)} |(G(y(r))^{-1}-G(y(s))^{-1})G(x)h|_Z
\\&&+|(\nabla^G\Phi_n(y(r))-\nabla^G\Phi_n(y(s)))G(y(s))^{-1}G(x)h|_V
\\&&\leq M|(G(y(r))^{-1}-G(y(s))^{-1})G(x)h|_V+|(\nabla^G\Phi_n(y(r))-\nabla^G\Phi(y(s)))G(y(s))^{-1}G(x)h|_V.
\end{eqnarray*}
Then,   using Assumption \ref{assGbis} and again  the fact that $\nabla^G\Phi_n\in C_m^s(U,\call(Z,V))$,
we see that $\varphi'(r)\to\varphi'(s)$ as  $r\to s$. By arbitrariness of $r\in [-1,1]$, we conclude that $\varphi'\in C^1([-1,1],V)$.
So, applying \cite[Prop.\,3.5,\,p.\,76]{Zeilder86-I}, we get
\begin{equation}\label{pplk}
\frac{\Phi_n(x+sG(x)h)-\Phi_n(x)}{s}= \frac{1}{s}\int_0^s
\varphi_n'(r)dr
= \frac{1}{s}\int_0^s
\nabla^G\Phi_n(y(r))G(y(r))^{-1}G(x)hdr, \ \ \ \forall s\in[-1,1]\setminus 0.
\end{equation}
%
Now, as $n\rightarrow\infty$, we have the convergences
\begin{align*}
\Phi_{n}(x+sG(x)h) &
\rightarrow\Phi(x+sG(x)h),\\
\Phi_{n}(x)&  \rightarrow\Phi(x)  ,\\
\nabla^G\Phi_n(y(r))G(y(r))^{-1}G(x)h  &  \rightarrow
A(y(r))G(y(r))^{-1}G(x)h, \ \ \ r\in \R.
\end{align*}
So, from \ref{pplk}, we get
\begin{eqnarray}\label{pplk2}
\frac{\Phi(x+sG(x)h)-\Phi(x)}{s}= \frac{1}{s}\int_0^s
A(y(r))G(y(r))^{-1}G(x)hdr.
\end{eqnarray}
As $A\in C_m^s(U,\call(Z,V))$, arguing as done above for $\Phi_n$,
we see that the function $[-1,1]\to \R, \ r\mapsto  A(y(r))(G(y(r))^{-1}G(x)h) $ is continuous. Then, from \eqref{pplk2} it follows
\begin{eqnarray}\label{pplk2bis}
\lim_{s\to 0}\frac{\Phi(x+sG(x)h)-\Phi(x)}{s}=
A(x)G(x)^{-1}G(x)h.
\end{eqnarray}
Let $k\coloneqq G(x)^{-1}G(x)h$ and
observe that $G(x)^{-1}G(x)k=k$ and  $G(x)h=G(x)k$, i.e. $h-k\in \mbox{ker}(G(x))$. Then,   we get from \eqref{pplk2bis} and \eqref{kerker} that
there exists $\nabla^G \Phi(x)$ and coincides with $A(x)$.
%
The convergence of $\Phi_n$ to $\Phi$ in the norm
$\left| \cdot\right| _{{\calg}_{m}^{1,G}}$ then follows, completing the proof.
\end{proof}

\begin{Corollary}
\label{cr4:DGclosed}
Let  $U$, $V$, $Z$ be three Banach spaces, let $G:X\to\call_u(Z,X)$, and let $m\ge 0$.
\begin{enumerate}[(i)]
\item
The linear unbounded operators
\[
\nabla^{G}: \mathcal{G}_{m}^{1,G}\left(  {U},V\right)  \subseteq
C_{m}\left({U},V\right)\rightarrow
C_{m}^{s}\left(  {U},\call\left(Z,V\right)  \right),
\]
\[
D^{G}: C_{m}^{1,G}\left(  {U},V\right)  \subseteq
C_{m}\left({U},V\right)\rightarrow
C_{m}\left({U},\call\left(Z,V\right)  \right),
\]
are closed.
\item Let
$f\in L^1([0,+\infty),
C_m\left(  {U},V\right))$ be  such that $f(t)\in \calg_m^{1,G}(U,V)$ (resp. $C_m^{1,G}(U,V)$) for a.e. $t\in [0,+\infty)$ and $\nabla^G f\in L^1([0,+\infty),C_m^s(U,\call(Z,V)))$ (resp., $D^G f\in L^1([0,+\infty),C_m(U,\call(Z,V)))$)
Then
$$
L\coloneqq \int_{0}^{+\infty}f(t)dt\in \calg_{m}^{1,G}\left({U},V\right),
$$
with the integral intended in Bochner sense in $C_m(U,V)$, and
$$
\nabla^{G}L =\int_{0}^{+\infty}\nabla^{G}f(t)dt \ \ \mbox{(resp.,} \  D^{G}L =\int_{0}^{+\infty}D^{G}f(t)dt),
$$
with the integral intended in Bochner sense in $C_m^s(U,\call(Z,V))$ (resp., in $C_m(U,\call(Z,V))$).
%
\end{enumerate}
\end{Corollary}

\begin{proof}
(i) This part is a straightforward consequence of Proposition \ref{lm4:GBanach}.

(ii)
The claim  follows from item (i)  and from
%
%
 \cite[Th.\,6,\,p.\,47]{DiestelUhl77}.
\end{proof}

When $G=I$, we drop the superscript $G$ in the notation for derivatives and  in all the spaces
introduced in this subsection.

 \begin{Remark}\label{rem:fr}
 We point out that, dealing with $G$-gradients, also  classical properties other than exchange of differentiation and integration are not obvious. For instance, consider the following classical property (see \cite[Prop.4.8(c),\,p.\,137]{Zeilder86-I}): if   $f:U\to V$ is Gateaux differentiable and $\nabla f:U\to \call(U,V)$ is continuous at $x\in U$, then   $f$ is Fr\'echet differentiable at $x$ and $Df(x)=\nabla f(x)$. If we want to extend this property to $G$-gradients, we must
strengthen Assumption \ref{assGbis}: for example a sufficient condition is to require  that
$$
\lim_{y\to x}\sup_{h\in\cald_G} \frac{|G^{-1}(y)G(x)h-h|}{|h|}=0.
$$
Without this assumption, the conclusion is not guaranteed.
\end{Remark}

\section{Mild solutions of semilinear elliptic equations in Hilbert spaces}
\label{sec:sem}

In this section we address  the problem of existence and uniqueness of solutions to the  semilinear elliptic  equation \eqref{eq4:HJbasicelliptic} in a separable Hilbert space $H$.
In order to cover important families of applied examples,
we consider a version  of  \myref{eq4:HJbasicelliptic} where the dependence of the Hamiltonian on the gradient is further specified. Precisely, given a separable Hilbert space $K$ and  $G: H\to \call_u (K,H)$, we consider a nonlinear term in the form $F(x,v(x),Dv(x))=F_0(x,v(x),D^{G}v(x))$, where
  $D^{G}$ denotes the $G$-derivative  defined in the previous section\footnote{Here we are  using the symbol $D^G$ only in a formal sense, without necessarily referring to the $G$-Fr\'echet derivative.}. Then,  \myref{eq4:HJbasicelliptic}  reads as
\begin{equation}
\label{eq4:HJbasicelliptic1}
\lambda v(x) -\frac{1}{2}\;\mbox{\rm Tr}\;[Q(x)D^2v(x)]
-\langle Ax+ b(x),Dv(x) \rangle  -F_0(x,v(x),D^{G}v(x))=0,
\quad  x \in H.
\end{equation}
%
{From now on, we assume throughout the whole paper that $G$ satisfies Assumption \ref{assGbis} with $U=H$ and $Z=K$. Due to that, the results of Propositions \ref{intder} and \ref{lm4:GBanach} hold true.}

We now introduce a formal argument to motivate the concept of \emph{mild solution} to
\myref{eq4:HJbasicelliptic1}.
Let us consider
the second order differential operator associated to the linear part of \eqref{eq4:HJbasicelliptic1}, that is the linear operator formally defined by
\begin{equation}\label{operatorA}
[\mathcal{A} \phi](x)\coloneqq \frac{1}{2}\;\mbox{\rm Tr}\;[Q(x)D^2\phi(x)]
+\< Ax+ b(x),D\phi(x) \>.
\end{equation}
Recall that, if $X$ is a Banach space, a family $(P_s)_{s\geq 0}\subseteq \call(X)$  is called a one parameter semigroup   if
$P_0=I$ and
$
P_{t}P_{s}=P_{t+s}
$ for every $s,t\ge 0 $.
 The operator $\cala$ defined in \eqref{operatorA} can be formally associated to a transition semigroup  $(P^\cala_s)_{s\geq 0}$ of linear operators in two ways.
\begin{enumerate}
\item
Through the stochastic differential equation (SDE)
\begin{equation}
\left\{
\begin{array}
[c]{ll}
dX\left(  s\right)  =\left[  AX\left(  s\right)
+b\left(X\left(s\right)\right)  \right]  ds
+\sigma\left(X\left( s\right)\right)  dW\left(s\right),
 & s\geq 0, \\
X\left(  0\right)  =x, & x\in H,
\end{array}
\right.  \label{eq4:SDEformalell}
\end{equation}
where $Q(x)\coloneqq \sigma(x)\sigma^*(x)$ for some $\sigma(x)\in \call(\Xi,H)$, where $\Xi$ is a separable Hilbert space (of course this is possible by taking $\Xi=H$ and $\sigma(x)= Q(x)^{1/2}$, and $W$ is a cylindrical Brownian motion in $\Xi$ (see \cite[Ch.\,4]{DaPratoZabczyk14}).
Then, assuming existence and uniqueness of solutions, in some sense, to \eqref{eq4:SDEformalell} and calling $X(\cdot,x)$ this solution for each given $x\in H$, one defines for $\phi\in B_m(H)$,
\begin{equation}\label{trsem}
P^\cala_{s}[\phi](x)\coloneqq \E[\phi (X(s,x)], \qquad  \forall (s,x)\in[0,+\infty)\times H.
\end{equation}

\item
Through the Kolmogorov equation
\begin{equation}
\label{eq4:HJbasicelliptic1bis}
\begin{cases}
u_t(t,x)=\frac{1}{2}\;\mbox{\rm Tr}\;[Q(x)D^2u(t,x)]
+\langle Ax+ b(x),Du(t,x) \rangle,
\quad  (t,x) \in [0,+\infty)\times H,\\
u(0,x)=\phi(x), \ \ x\in H.
\end{cases}
\end{equation}
Then, assuming well-posedness (existence and uniqueness of solution, in some sense) to \eqref{eq4:SDEformalellbis}, calling $u_\phi$ this solution, one defines for $\phi\in B_m(H)$
\begin{equation}\label{trsembis}
P^\cala_{s}[\phi](x)\coloneqq u_\phi(s,x), \qquad  \forall (s,x)\in[0,+\infty)\times H.
\end{equation}
\end{enumerate}
If $(P^\mathcal{A}_s)_{s\geq 0}$ was a $C_0$-semigroup, e.g. in $C_m(H)$ --- that is, other than the semigroup properties, also $\lim_{s\to 0^+}P^\cala_s\phi =\phi$ holds for every $\phi\in C_m(H)$ --- and if $\mathcal{A}$ was its generator (see \cite[Ch.\,II]{EngelNagel99}), it would hold the classical representation of the resolvent operator as Laplace transform of the semigroup (see \cite[Ch.\,II,\,Th.\,1.10]{EngelNagel99}): for all $\lambda$ large enough, the operator $\lambda I- \cala:\cald(\cala)\rightarrow H$ is bijective with bounded inverse and  \begin{equation}\label{eq:4resolventintegral1bis}
(\lambda I- \cala)^{-1}[g]= \int_0^{+\infty} e^{-\lambda s} P^\cala_s [g] ds, \ \ \forall g\in C_m(H),
\end{equation}
with the integral intended in Riemann sense in the space  $C_m(H)$.
Unfortunately,
$(P^\cala_s)_{s\geq 0}$ is not, in general, a $C_0$-semigroup in $C_m(H)$ or in many other functional spaces.  Indeed, in the framework of spaces of functions not vanishing at infinity, the $C_0$-property fails even in  basic cases. For instance, this property fails in the case of Ornstein-Uhlenbeck semigroup in $C_b(\mathbb{R})$ (see, e.g.,  \cite[Ex.\ 6.1]{Cerrai94} for a counterexample in $UC_b(\R)$, or  \cite[Lemma\ 3.2]{DaPratoLunardi95},
 which implies this semigroup is $C_0$ in $UC_b(\R)$ if and only if the drift of the SDE vanishes). Even worse: given $\varphi\in C_b(H)$,  the map $[0,+\infty)\to C_b(H)$, $t\mapsto P_t^\cala\varphi$ is not in general measurable, as shown in Example \ref{OU-notmes}; this prevents to intend the integral in \eqref{eq:4resolventintegral1bis} in Bochner sense in the space $C_b(H)$. Nevertheless, to some extent, one can still consider the operator $\mathcal{A}$ as a kind of generator for $(P^\cala_s)_{s\geq 0}$ and prove a pointwise $C_0$-property and a pointwise  counterpart of \eqref{eq:4resolventintegral1bis}: that is, one can prove (see, e.g., \cite{Cerrai94, Priola99}) that, for each fixed
$\phi \in C_m(H)$ and $x\in H$, it is $\lim_{s\to 0^+}P^\cala_s[\phi](x)=\phi(x)$ and for all $\lambda$ large enough
\begin{equation}\label{eq:4resolventintegral1}
(\lambda I- \cala)^{-1}[g](x)= \int_0^{+\infty} e^{-\lambda s} P^\cala_s [g](x) ds.
\end{equation}
Then, the solution of the linear equation
$(\lambda - \cala) u =g$ is
\begin{equation}\label{eq:4resolventintegral}
u(x)=(\lambda I- \cala)^{-1}[g](x)= \int_0^{+\infty} e^{-\lambda s} P^\cala_s [g](x) ds, \quad x\in H.
\end{equation}
Now,
(\ref{eq4:HJbasicelliptic1}) can be rewritten as
$$
\lambda v(x)- (\cala v)(x) =F_0(x,v(x),D^G v(x)),
 \quad x \in H.
$$
Then, taking $g(x)=F_0(x,v(x),D^Gv(x))$ and  applying \myref{eq:4resolventintegral}, we can rewrite  (\ref{eq4:HJbasicelliptic1}) as an integral equation:
\begin{equation}
v(x)= \int_0^{+\infty }e^{-\lambda s}
P^\cala_s \left[  F_0(\cdot,v\left(\cdot\right),D^{G}v\left(\cdot\right))\right](x)ds,
\quad x\in H. \label{eq4:solmildHJBell}
\end{equation}
The concept of mild solution to \eqref{eq4:HJbasicelliptic1} relies on the last integral form  \eqref{eq4:solmildHJBell}.
\begin{Definition}
\label{df4:solmildHJBell} Let $m\geq 0$ and $G:H\to \call_u(K,H)$, with $K$ Hilbert space. We say that a
function $v:H\to \R$ is a mild
solution to (\ref{eq4:HJbasicelliptic1}) in the space $\calg_m^{1,G}(H)$ (respectively, $C _{m}^{1,G}\left(H\right)$) if it belongs to $\calg_m^{1,G}(H)$ (respectively, $C _{m}^{1,G}\left(H\right)$) and
 (\ref{eq4:solmildHJBell}) holds for every $x \in H$.
\end{Definition}

Let us start with the assumptions on the Hamiltonian function $F_0$.

\begin{Assumption}\label{hp4:F0ell}
\begin{enumerate}[(i)]
\item[]
\item\label{F01} $F_0:H\times \R\times (K,\tau^K_w)\to\R$ is sequentially continuous.

\item\label{F02} There exists $L>0$ such that
\[
\left|  F_0(x,y_{1},z_{1})  -F_0(x,y_{2},z_{2})  \right|
\leq L\left(  |y_{1}-y_{2}| +|z_{1}-z_{2}|_{K}\right), \ \ \forall x\in H, \ \forall y_{1},y_{2}\in\mathbb{R}, \ \forall  z_{1},z_{2}\in K.
\]

\item\label{F03} There exists $m\geq 0$ and  $L^{\prime}>0$  such that
\[
\left|F_0(x,y,z)\right|
\leq L^{\prime}\left(1+|x|_H^m+|y| +|z|_{K}\right)  ,\ \ \forall x\in H, \ \forall y\in\mathbb{R}, \ \forall  z\in K.
\]

\end{enumerate}
\end{Assumption}

\begin{Assumption}\label{hp4:F0ellbis}
\begin{enumerate}[(i)]
\item[]
\item\label{F01bis} $F_0:H\times \R\times K\to\R$ is continuous.
\item Assumptions \ref{hp4:F0ell}(ii)-(iii) hold.
%
%
%

\end{enumerate}
\end{Assumption}

\begin{Remark}
We notice that the assumption of sequential continuity of $F_0$ in the last variable in Assumption \ref{hp4:F0ell}
(\ref{F01}) reduces to the assumption of continuity when $K$ is finite-dimensional, which covers, in particular, Hamiltonian functions  depending only on finite-dimensional   projections of the gradient. This arises, typically, in control problems with delays (see \cite{Federico11, FedericoGoldysGozzi10SICON, FedericoGoldysGozzi11SICON, FedericoTacconi14, GozziMarinelli06, GozziMasiero12}.
\end{Remark}

We now formulate assumptions directly on $(P^\cala_s)_{s\geq 0}$ (whatever is the way to define it, by \eqref{eq4:SDEformalell} or \eqref{eq4:HJbasicelliptic1bis}) and on $F_0$ that allow to solve (\ref{eq4:solmildHJBell})  under the restriction that $\lambda$ has to be large enough. In order to simplify the notation we write $P_s$ for $P_s^\cala$. One should keep in mind that, in the following,  $P_s$ is always an object associated to $\cala$ through \eqref{eq4:SDEformalell} or \eqref{eq4:HJbasicelliptic1bis}. Nevertheless, under such assumptions,  the integral equation \eqref{eq4:solmildHJBell} could be seen as the mild form of different, possibly more general semilinear
 equations, e.g. when $(P_s)_{s\geq 0}$ is associated to more general processes, e.g., L\'evy
processes, in which case the operator $\cala$ is integro-differential.

We provide two different sets of assumptions  concerning the semigroup $(P_s)_{s\geq 0}$.
The last three requirements in each of the next two sets of assumptions concern smoothing properties of the semigroup: it is required that continuous functions are mapped by the operator $P_s$ into $G$-differentiable ones for each $s>0$ and that some related issues concerning measurability and growth are fulfilled too.
Let
\begin{align}
\cali:= & \biggl\{
\eta\in L^1_{loc}([0,+\infty),\R_+): \
 \eta \hbox{ is bounded in a neighborhood of $+\infty$ }
\biggr\}.
\label{eq:defI}
\end{align}
\begin{Assumption}\label{hp4:Psemigroupell}
Let $m\geq 0$ be the constant of Assumption \ref{hp4:F0ell}.
\begin{enumerate}[(i)]
\item\label{hp4:Psemigroupell2}
The family
$(P_{s})_{s\geq 0}$ is  a semigroup of continuous linear operators in the space $C_m(H)$.
\item The map $[0,+\infty)\times H\to \R, \ (s,x)\mapsto P_s[\phi](x)$ is measurable for every $\phi\in C_m(H)$.
\item\label{assi} There exist constants $C>0$ and $a\in \R$ such that
$$
|P_s[\phi](x)|\leq Ce^{as}|\phi|_{C_m(H)}(1+|x|^m), \ \ \ \forall s\geq 0, \ \forall x\in H, \ \forall \phi\in C_m(H),
$$
i.e., $\left|  P_{s}\right|_{\call(C_m(H))}  \leq Ce^{a s}$ for every $s\geq 0$.
\item
\label{hp4:smoothingGell}
 $P_{s}(C_m(H))\subseteq \calg_m^{1,G}(H)$
 for every $s>0$.
\item\label{ip5} The map $(0,+\infty)\times H\to K, \ (s,x)\mapsto \nabla^GP_s[\phi](x)$ is measurable for every $\phi\in C_m(H)$.
\item\label{ip4}
There exists $\gamma_G\in \cali$ and $a_G\in \R$ such that
$$
\left|\nabla^{G}P_{s}[\phi](x)\right|_K  \leq\gamma_G(s)e^{a_G s} |\phi|_{C_m(H)}(1+|x|^m), \ \ \ \forall \phi\in C_m(H), \ \forall x\in H,  \ \forall s>0,$$
i.e.,
$|\nabla^GP_s|_{\call(C_m(H),C_m^s(H,K))}\leq\gamma_G(s)e^{a_Gs}$ for every $s>0$.
%
%

\end{enumerate}
\end{Assumption}

\smallskip

\begin{Assumption}\label{hp4:Psemigroupellbis}Let $m\geq 0$ be the constant of Assumption \ref{hp4:F0ellbis}.
\begin{enumerate}[(i)]
\item\label{hp4:Psemigroupell2cont}
Assumptions \ref{hp4:Psemigroupell}(i)--(iii) hold.

%
\item
\label{hp4:smoothingGellbis}
 $P_{s}(C_m(H))\subseteq C_m^{1,G}(H)$
 for every $s>0$.
\item\label{ip5bis}
The map $(0,+\infty)\times H \to K, \ (s,x) \mapsto D^G P_{s}[\phi](x)$ is measurable.
\item\label{ip4bis}
There exists $\gamma_G\in \R$ and $a_G\in \R$
such that
$$
\left|D^{G}P_{s}[\phi](x)\right|_K  \leq\gamma_G(s)e^{a_G s} |\phi|_{C_m(H)}(1+|x|^m), \ \ \ \forall \phi\in C_m(H), \ \forall x\in H,  \ \forall s>0,$$
i.e.,
$|D^GP_s|_{\call(C_m(H),C_m^s(H,K))}\leq\gamma_G(s)e^{a_Gs}$ for every $s>0$.

\end{enumerate}
\end{Assumption}
\begin{Remark}
About the growth estimate on the semigroup in Assumptions \ref{hp4:Psemigroupell} and \ref{hp4:Psemigroupellbis}, we note that we generically require $a\in\R$. As matter of fact, for transitions semigroups, which are the ones we are interested in, it is always $a\geq 0$.
Nevertheless, requiring $a\geq 0$ is not necessary for the arguments of Theorem \ref{th4:exunHJBgeneralell} below. We keep the (a priori weaker) assumption
$a\in \R$, which may be verified in other cases (e.g. for semigroups of negative type), and may give rise to a sharper result in Theorem \ref{th4:exunHJBgeneralell}, guaranteeing the conclusion also (possibly) for negative $\lambda$.
\end{Remark}

We state and prove now our main result.
\begin{Theorem}
\label{th4:exunHJBgeneralell}
\begin{enumerate}[(i)]
\item[]
\item Let Assumptions \ref{hp4:Psemigroupell} and \ref{hp4:F0ell} hold. Then,  there exists $\lambda_0\in \R$ such that, for every $\lambda > \lambda_0$, there exists
 a unique mild solution $v\in \calg_m^{1,G}(H)$ to (\ref{eq4:HJbasicelliptic1}).
\item Let Assumptions \ref{hp4:Psemigroupellbis} and \ref{hp4:F0ellbis} hold. Then there exists $\lambda_0\in \R$ such that, for every $\lambda > \lambda_0$, there exists
 a unique mild solution $v\in C_m^{1,G}(H)$ to (\ref{eq4:HJbasicelliptic1}).
 \end{enumerate}
%
\end{Theorem}
%
%

\begin{proof}

{\em Proof of (i)}.
Consider the product Banach space
$ C _{m}(H)\times  C_{m}^s(H,K)$
endowed with the product norm $|\cdot|_{C_m(H)}\times |\cdot|_{C_m^s(H,K)}$. We define,
for  $(u,v)\in  C_{m}(H)\times  C^s_{m}(H,K)$,
\begin{equation}
\Upsilon_{1}\left[  u,v\right]  (x)\coloneqq \int_0^{+\infty}e^{-\lambda s}
P_{s}\left[F_0(\cdot,u(\cdot),v(\cdot))  \right](x)ds, \ \ x\in H, \label{eq4:Lambda1ell}
\end{equation}
\begin{equation}
\Upsilon_{2}\left[  u,v\right]  (x) \coloneqq \int_{0}^{+\infty}e^{-\lambda s}
\nabla^{G}P_{s}\left[F_0(\cdot,u(\cdot),v(\cdot))\right](x)ds, \ \ x\in H.
\label{eq4:Lambda2ell}
\end{equation}
We now accomplish the proof in three steps. In the rest of the proof, $C$, $a$, $a_G$, and $\gamma_G$ are the objects appearing in Assumption \ref{hp4:Psemigroupell}.

\smallskip

{\bf 1.} {\em $\Upsilon=(\Upsilon_1,\Upsilon_2)$ is well defined as a map from $C_{m}(H)\times C^s_{m}(H,K)$ into itself for all $\lambda>a\vee a_G$.}

Let $\lambda >a\vee a_G$ and let $(u,v)\in C _{m}(H)\times  C^s_{m}(H,K)$.
The fact that $v\in C_m^s(H,K)\cong C_m^s(H,K^*)$ means that  $v:(H,|\cdot|_H)\to (K,\tau_w^K)$ is continuous. Hence,
setting $\psi(\cdot):=F_0\left(\cdot,u(\cdot),v(\cdot)\right)$, we get $\psi\in C_m(H)$   by Assumption \ref{hp4:F0ell}(\ref{F01}) and (\ref{F03}).
Then, by Assumption \ref{hp4:Psemigroupell}(ii)-(\ref{assi}), $\Upsilon_1[u,v](x)$ is well defined for every $x\in H$ and every $\lambda>a$.
Finally, $\Upsilon_1[u,v]\in C_m(H)$ by dominated convergence theorem and by Assumption \ref{hp4:Psemigroupell}(\ref{assi}).

Concerning  $\Upsilon_2[u,v]$, first of all we note that  $\Upsilon_2[u,v](x)$ makes  sense as Bochner integral in $K$ by Assumption \ref{hp4:Psemigroupell}(\ref{hp4:smoothingGell})--(\ref{ip4})
%
for every $x\in H$ and every $\lambda >a_G$.
Moreover, in view  of the same assumptions, by dominated convergence theorem and by \eqref{ppppk} we conclude that $\Upsilon_2[u,v] \in C^s_{m}(H,K)$.

\smallskip

{\bf 2.} {\em $\Upsilon$ is contraction in
$ C_{m}(H)\times  C_{m}^s(H,Z)$ for $\lambda$ sufficiently large.}

\noindent
Let $( u_{1},v_{1}), ( u_{2},v_{2})\in  C_{m}(H)\times  C^s_{m}(H,K)$.
By  Assumption \ref{hp4:Psemigroupell}(\ref{assi})
and  Assumption \ref{hp4:F0ell}(\ref{F02}), we have
\begin{align*}
&  \left|  \Upsilon_{1}\left[  u_{1},v_{1}\right]
-\Upsilon_{1}\left[  u_{2},v_{2}\right]\right|_{C_m(H)}
\\
&   \leq \sup_{x\in H}\frac{\left|\int_{0}^{+\infty}e^{-\lambda s} P_{s}
\left[F_0(\cdot,u_{1}(\cdot),v_{1}(\cdot))
-F_0(\cdot,u_{2}(\cdot),v_{2}(\cdot))\right](x) ds\right|}{1+|x|^m}
\\
&  \leq \int_{0}^{+\infty}e^{-\lambda s}\sup_{x\in H} \frac{\left|P_{s}
\left[F_0(\cdot,u_{1}(\cdot),v_{1}(\cdot))
-F_0(\cdot,u_{2}(\cdot),v_{2}(\cdot))\right](x)\right|}{1+|x|^m} ds
\\
& \le  \int_{0}^{+\infty}C e^{-(\lambda-a) s}
\left|F_0(\cdot,u_{1}(\cdot),v_{1}(\cdot))
-F_0(\cdot,u_{2}(\cdot),v_{2}(\cdot))\right|_{ C _m(H)} ds
\\
&\le
\frac{C}{\lambda -a}
L\left(|u_{1} -u_{2}|_{ C_{m}(H)}+|v_{1} -v_{2}|_{ C^s _m(H,K)}\right).
\end{align*}
On the other hand, similarly, using  Assumption \ref{hp4:Psemigroupell}(\ref{ip4})
and Assumption \ref{hp4:F0ell}(\ref{F02}), we have
\begin{align*}
& \left|\Upsilon_2 [u_1,v_1]- \Upsilon_2 [u_2,v_2]\right|_{ C^s_m(H,K)}
\le
L\left(\int_{0}^{+\infty}e^{-(\lambda-a_G) s}\gamma_G(s)ds\right)
\left(|u_{1} -u_{2}|_{C _{m}(H)}+|v_{1} -v_{2}|_{ C_m^s(H,K)}\right).
\end{align*}
We conclude
\begin{align*}
&\left|\Upsilon_{1}[u_1,v_1]-\Upsilon_{1}[u_2,v_2] \right|_{ C_{m}(H)}
+\left|\Upsilon_{2}[u_1,v_1]-\Upsilon_{2}[u_2,v_2]\right|_{ C_m^s(H,Z)}\\
&  \leq
L\left[\frac{C}{\lambda -a} +
\int_{0}^{+\infty}e^{-(\lambda-a_G) s}\gamma_G(s)ds \right]
\left[  \left|  u_{1}-u_{2}\right|_{ C_{m}(H)}
+\left|v_{1}-v_{2}\right|_{ C^s_{m}(H.Z)}\right].
\end{align*}
Setting
$$
\alpha(\lambda):=L\left[\frac{C}{\lambda -a} +
\int_{0}^{+\infty}e^{-(\lambda-a) s}\gamma_G(s)ds \right]
$$
we see that $\alpha:(a\vee  a_G, +\infty)\to (0,+\infty)$ is strictly decreasing and continuous,  and that  $\lim_{\lambda\to a^+}C(\lambda)=+\infty$ and $\lim_{\lambda\to \,\, + \infty}C(\lambda)=0$.
Hence, letting $\lambda_0\coloneqq \inf\{\lambda>a\vee a^G: \ \alpha (\lambda)\leq 1\}$, with the convention $\inf\emptyset =a\vee a_G$,
we get that the map $\Upsilon$
is a contraction  for each for $\lambda>\lambda_0$ and, therefore, admits a unique fixed point in
$ C_{m}(H)\times  C_m^s(H,K)$ for such values of $\lambda$.

\smallskip

{\bf 3.} {\em The first component of the
fixed point of $\Upsilon$ is the unique mild solution of \myref{eq4:HJbasicelliptic1} for $\lambda>\lambda_0$. }

By Proposition  \ref{intder} and Assumption \ref{hp4:Psemigroupell}(ii)-(\ref{assi}),
$\Upsilon_{1}\left[  u,v\right]$ is $G$-Gateaux  differentiable on $H$ for all $(u,v) \in  C_{m}(H)\times  C _m^s(H,K)$ and
$\Upsilon_{2}\left[  u,v\right]=\nabla^{G}\Upsilon_{1}\left[  u,v\right]$.

Let now $[\overline u,\overline v]$ be the fixed point of $\Upsilon$, so
$\Upsilon[\overline u,\overline v]  =(\overline u,\overline v)$.
It follows, from what said just above, that $\overline{u}$ is $G$-Gateaux differentiable and
$\nabla^{G}\overline{u}=\overline{v}$, hence
$\nabla^{G}\overline{u}\in  C_{m}(H,K)$.
Finally, replacing $\overline v$ by $\nabla^G \overline u$   into the definition of $\Upsilon_1$, we get
$$\overline{u}(x)=\int_0^{+\infty}e^{-\lambda s}P_{s}
\left[F_0(\cdot,\overline{u}(\cdot),\nabla^{G}\overline{u}(\cdot))\right](x)ds
$$
The above imply that $\overline u$ is a mild solution of \myref{eq4:HJbasicelliptic1}.

We now prove uniqueness.
Let $u^{\ast}$ be another mild solution to equation (\ref{eq4:HJbasicelliptic1}). Then $u^{\ast}$ is $G$-Gateaux differentiable with $\nabla^G u^*\in  C^s _m(H,K)$.
Hence, setting $v^*:=\nabla^G u^*$, we  see that $(u^*,v^*)$
is a fixed point of $\Upsilon$ and so, by uniqueness,  $u^*=\overline u$ and $v^*=\overline v$.
This completes the proof.

\medskip

{\em Proof of (ii)}.
The proof of this claim  works similarly to  the proof of {\em (i)}: one just needs to
perform the fixed point argument for the map $\Upsilon$
in the space $C_m(H)\times C_{m}(H,U)$ replacing $\nabla^G$ by $D^G$.
In step 1 the difference is that we need to prove that $\Upsilon$ maps the space
$C_m(H)\times C_{m}(H,U)$ in itself.
Indeed, given $(u,v)\in C_m(H)\times C_{m}(H,U)$ the function $\psi(\cdot)=F(\cdot,u(\cdot),v(\cdot))$ is continuous as a  composition of continuous functions (Assumption \ref{hp4:F0ellbis}-(i)). Then, the fact that
$\Upsilon [u,v] \in C_m(H)\times C_{m}(H,U)$ follows, similarly to point (i),
from Assumption \ref{hp4:Psemigroupellbis}.
Steps 2 and 3 are performed exactly in the same way and we omit them.
\end{proof}

\begin{Remark}

The fact that an existence/uniqueness theorem
holds under very general assumptions on the data only if $\lambda$ is large enough is a structural issue, arising  also with other concepts of solution.
However, under suitable additional assumptions (see \cite{Cerrai01, Cerrai01-40, GozziRouy96}), by using monotone operator techniques one
can extend such a result to each $\lambda >0$.
\end{Remark}
%
\subsection*{Strong Feller case}
Assumption \ref{hp4:F0ell}(i) is not verified
in some concrete cases, e.g. when $F_0$ depends on the norm of the last variable.
A way to overcome this problem is to  require more on the semigroup
$(P_{s})_{s \ge 0}$, assuming that it a semigroup of bonded linear operators on $B_m(H)$ and it is  strongly Feller,  i.e.
\begin{equation}\label{strongfeller}
P_{s}(B_m(H))\subseteq C_m(H), \qquad \forall s > 0.
\end{equation}
This assumption is verified in many important examples and in  this case one can prove the following second main result.
\begin{Theorem}\label{Th:strongfeller}
\begin{enumerate}[(i)]
\item[]
\item Let Assumption \ref{hp4:Psemigroupell} hold replacing $C_m(H)$ by $B_m(H)$ everywhere (so, in particular,  \eqref{strongfeller} holds).
 Let $F_0:H\times\R\times K\to \R$ be measurable and let  \ref{hp4:F0ell}(ii)-(iii) hold. Then,  there exists $\lambda_0\in \R$ such that, for every $\lambda > \lambda_0$, there exists
 a unique mild solution $v\in \calg_m^{1,G}(H)$ to (\ref{eq4:HJbasicelliptic1}).
\item Let Assumption \ref{hp4:Psemigroupellbis} hold  replacing $C_m(H)$ by $B_m(H)$ everywhere (so, in particular,  \eqref{strongfeller} holds).
  Let $F_0:H\times\R\times K\to \R$ be measurable and let  \ref{hp4:F0ell}(ii)-(iii) hold. Then there exists $\lambda_0\in \R$ such that, for every $\lambda > \lambda_0$, there exists
 a unique mild solution $v\in C_m^{1,G}(H)$ to (\ref{eq4:HJbasicelliptic1}).
 \end{enumerate}
\end{Theorem}
\begin{proof} It follows by the same arguments of the proof of  Theorem  \ref{th4:exunHJBgeneralell}, once we appropriately choose the spaces
where we apply the fixed point argument.
\end{proof}

 \begin{Remark}
\label{rm4:Fellerbis}
\begin{itemize}
\item[]
\item[(i)]
In \cite{DaPratoZabczyk02} and in \cite{Cerrai01} the authors
prove existence and uniqueness of the mild solution, in the case $G=I$,
performing the fixed point theorem in a different product space. Basically, with respect to ours, {the product space considered in the aforementioned references} is more regular in the first component (uniformly continuous functions) and less regular in the second component (bounded and measurable functions).
{We notice that, also in our case of general $G$, it is possible} to prove a version of Theorems \ref{th4:exunHJBgeneralell} and \ref{Th:strongfeller} {in the latter product space}.

\item[(ii)] If $F_0$ is only measurable and the strong Feller property does not hold, one can define the concept of mild solution in spaces of measurable functions and prove  results similar to the ones of Theorem \ref{th4:exunHJBgeneralell}.
\end{itemize}
\end{Remark}

\section{$G$-smoothing properties of transition semigroups}\label{S4:SMOOTHING}
In this section we provide special cases of transition semigroups satisfying Assumption \ref{hp4:Psemigroupell} or \ref{hp4:Psemigroupellbis}.
We will deal with the case when $(P_s)_{s\geq 0}$ is defined through the solution of an SDE  in a
  filtered probability space  $(\Omega, \mathcal{F}, (\mathcal{F}_t)_{t\geq 0}, \P)$ satisfying the usual conditions. Throughout the section $H,\Xi, K$ are real separable Hilbert spaces and $(W_t)_{t\geq 0}$ is a  cylindrical Wiener process $W$ with values in $\Xi$ defined in  the filtered probability space above.

Let  $\xi\in L^p(\Omega,\calf_0,\P;H)$ for some $p\geq 0$ and consider the SDE
\begin{equation}
\left\{
\begin{array}
[c]{ll}
dX\left(  s\right)  =\left[  AX\left(  s\right)
+b\left(X\left(s\right)\right)  \right]  ds
+\sigma\left(X\left( s\right)\right)  dW\left(s\right),
 & s\geq t, \\
X\left(  0\right)  =\xi.
\end{array}
\right.  \label{eq4:SDEformalellbis}
\end{equation}
\begin{Assumption}\label{ch1:ipotesiunobis}
\item[]
\begin{enumerate}[(i)]
  \item  $A:\cald(A)\subseteq H\rightarrow H$ is a closed densely defined (possibly unbounded) linear operator generating a $C_0$-semigroup of linear operators  $(e^{sA})_{s \ge 0}\subseteq \call(H)$.
  \item
The map  $b:H\rightarrow H$ is Lipschitz continuous: there exists $L\geq 0$ such that
\begin{equation}
\label{eq:bb0b}
  |b(x)-b(y) |\leq L
 |x-y|, \qquad \  \forall x,y\in H.
\end{equation}

  \item The map $\sigma: H \to \mathcal{L}(\Xi,H)$ is such that
  \begin{itemize}
  \item[-]
  for every $v\in \Xi$ the map  $(s,x)\mapsto e^{sA}\sigma(x)v$ is measurable;
  \item[-] $e^{sA}\sigma(x)\in \mathcal{L}_2(\Xi,H)$ for every $s> 0$, $x\in H$;
   \item[-] there  exists  $f\in L^2_{loc}([0,+\infty), \mathbb{R})$ such that
\begin{equation}\label{primaipotesisug1}
|e^{sA}\sigma(x)|_{\mathcal{L}_2(\Xi,H)}\leq  f(s)
  (1+|x|), \ \ \ \forall s> 0, \ \forall x\in H,
  \end{equation}
  and
  \begin{equation}
  |e^{sA}(\sigma(x)-\sigma(y))|_{\mathcal{L}_2(\Xi,H)}\leq
  f(s)|x-y|, \ \ \ \ \forall s> 0, \ \forall x,y\in H.
  \label{primaipotesisug2}
\end{equation}
\end{itemize}
\end{enumerate}
\end{Assumption}
We have the following classical result (see \cite[Th.\,7.5,\, 9.1,\,9.14]{DaPratoZabczyk14}, \cite[Sec.\,3.10]{GawareckiMandrekar10}), and \cite{Cosso15}.

\begin{Theorem}\label{th1:exSDEmult}
Let Assumption \ref{ch1:ipotesiunobis} hold. Then,
 for every $\xi\in L^p(\Omega,\calf_t,\P;H)$, with $p\geq 2$, SDE \eqref{eq4:SDEformalellbis}
 admits a unique mild solution in the space $\calh^{p,loc}_\calp(H)$, that is there exists a unique (up to modification) process in $\calh^{p,loc}_\calp(H)$, denoted by $X(\cdot,\xi)$, such that
\begin{equation}
X\left(t,\xi\right)  =e^{tA} \xi+ \int_0^te^{(t-s)A} b\left(X\left(s,\xi\right)\right)  ds
+\int_0^te^{(t-s)A}\sigma\left(X\left( s,\xi\right)\right)  dW\left(s\right), \ \ \ \forall t \geq 0.  \label{eq4:SDEformalellbisinteo}
\end{equation}
Such solution is a time-homogeneous Markov process and, for each $p>0$, there exist constants $C_p>0$ and $\alpha_p\in\R$ (see Remark \ref{rem:est} below) such that
\begin{equation}\label{eq:growXbeforeter}
\E[|X(t,\xi)|^p]\leq C_pe^{\alpha_p t} (1+\E|\xi|^p), \ \ \  \ \forall t\geq 0, \ \ \ \forall \xi\in L^p(\Omega,\calf_0,\P;H),
\end{equation}
\begin{equation}\label{eq:growXbeforeterbis}
\E[|X(t,\xi)-X(t,\eta)|^p]\leq C_pe^{\alpha_p t} \E|\xi-\eta|^p), \ \ \  \forall t\geq 0, \ \ \forall \xi, \eta\in L^p(\Omega,\calf_0,\P;H),
\end{equation}
\begin{equation}\label{eq:growXbeforetertris}
\E[|X(t,\xi)-\xi|^2]\leq C_2 (t+\E|e^{tA}\xi-\xi|), \ \ \  \forall t\in[0,1], \ \forall \xi\in L^p(\Omega,\calf_0,\P;H).
\end{equation}
%
Finally, if there exists $\gamma \in (0,1/2)$ such that
\begin{equation}\label{eq4:forcontdiffcoeff}
    \int_0^1 s^{-2\gamma}f^2(s)ds<+\infty,
\end{equation}
then (a version of) $X(\cdot,x)\in \mathcal{K}_\calp^{p,loc}(E)$ for each $p\geq 2$.
\end{Theorem}

\begin{Remark}\label{rem:est}
The estimates \eqref{eq:growXbeforeter}--\eqref{eq:growXbeforetertris} are stated for $p\geq 2$ and for finite horizon in \cite{DaPratoZabczyk14, GawareckiMandrekar10}. They can be extended to the case $p\in(0,2)$ by Jensen's inequality. Moreover, the exponential dependence in time of the estimates can be proved by an induction argument  starting from the estimate holding for fixed time horizon $T>0$ and  exploiting the time-homogeneity of the SDE.
Here we briefly describe the argument focusing on \eqref{eq:growXbeforeter}. By the aforementioned references we know that
\begin{equation}\label{eq:growXbeforeterq}
\E[|X(t,\xi)|^p]\leq c_p (1+\E|\xi|^p), \ \ \  \ \forall t\in[0,1], \ \ \ \forall \xi\in L^p(\Omega,\calf_0,\P;H).
\end{equation}
On the other hand, by  time-homogeneity of \eqref{eq4:SDEformalellbisinteo} and Markov property we have
$$\E\left[|X(t,\xi)|^p\right] =\E \left[|X(t-1, X(1,\xi))|^p\right], \ \ \forall t\in(1,2].$$
It follows
\begin{align*}
\E\left[|X(t,\xi)|^p\right] \leq c_p +c_p^2(1+\E|\xi|^p)), \ \ \forall t\in(1,2].
\end{align*}
Arguing by induction we get
\begin{align*}
\E\left[|X(t,\xi)|^p\right] &\leq (c_p +\ldots+ c_p^{n+1})+ c_p^{n+1}\E|\xi|^p)
\\&\leq n(c_p\vee 1)^{n+1} +c_p^{n+1}\E|\xi|^p
\\& = ne^{(n+1)\log (c_p\vee 1)}+e^{(n+1)\log c_p}\E|\xi|^p, \ \ \ \ \ \forall t\in(n,n+1], \ \forall n\in\N.
\end{align*}
Then \eqref{eq:growXbeforeter} follows by suitably defining $C_p$ and $\alpha_p$.
\end{Remark}


By Theorem \ref{th1:exSDEmult},  we see that, under Assumption \ref{ch1:ipotesiunobis}, the formula
\begin{equation}\label{trsemm}
P_{s}[\phi](x)\coloneqq \E[\phi (X(s,x))], \qquad  s \ge 0,
\end{equation}
defines a one parameter transition semigroup  $(P_s)_{s\geq 0}$ in the space $C_m(H)$ for every $m\geq 0$,
satisfying
Assumptions \ref{hp4:Psemigroupell}(i)--(iii) (hence,  Assumption \ref{hp4:Psemigroupellbis}(i)).
We are now going to study some special cases for which also the rest of  Assmuption \ref{hp4:Psemigroupell} or \ref{hp4:Psemigroupellbis} are satisfied by $(P_s)_{s\geq 0}$ defined in \eqref{trsemm}.

\subsection{The case of the Ornstein-Uhlenbeck semigroup}
\label{SS4:SMOOTHINGOU}
Let  $\mu\in H$ and $Q\in \call_1^+(H)$. We denote by
 $\mathcal{N}\left(\mu ,Q\right)$
 the Gaussian measure in $H$ with mean $\mu$ and covariance operator $Q\in \call_1^+(H)$ in the space $H$ (see  \cite[Ch.\,1]{DaPratoZabczyk02}). When $\mu=0$, we use the notation
 $\mathcal{N}\left(0,Q\right)=
\mathcal{N}_Q$.

Let us consider  the Ornstein-Uhlenbeck process, i.e. let us consider the special case when SDE  \eqref{eq4:SDEformalell} takes the form
\begin{equation}
\begin{cases}
dX(s)=AX(s) ds
+ \sigma dW(s), \ \ \ s\geq 0,\\
\label{eq4:SE1ellbis}
X(0)=x.
\end{cases}
\end{equation}
We deal under the following assumption.
\begin{Assumption}
\label{hp4:ABQforOU}
\begin{enumerate}
\item[]
\item[(i)] The linear operator $A$ is the generator of a $C_0$-semigroup $(e^{t A})_{t\geq 0}  $ in the Hilbert space $H$.

\item[(ii)]
 $\sigma\in \call(\Xi,H)$, \ $e^{sA}\sigma\sigma^{\ast}e^{sA^{\ast}}\in \call_1(H)$ for all $s >0$,
and
$$
\int_{0}^{t} {\rm Tr}\left[ e^{sA}\sigma\sigma^{\ast}e^{sA^{\ast}}\right]ds
<+\infty, \ \ \ \ \forall t\geq 0.
$$
\end{enumerate}
\end{Assumption}
We recall (see, e.g.,  \cite[Ch.\,II,\,Th.\,1.10]{EngelNagel99}) that, if  $(e^{tA})_{t\geq 0}$ is a $C_0$-semigroup in $H$, then there exist constants   $M\geq 1$ and $\omega\in \R$ such that
\begin{equation}\label{stimasem}
|e^{tA}|_{\call(H)}\leq Me^{\omega t}, \ \ \ \forall t\geq 0.
\end{equation}

Assumption \ref{hp4:ABQforOU} enables us, in particular, to   apply  Theorem \ref{th1:exSDEmult}
 and get the existence of a unique mild solution to equation \myref{eq4:SE1ellbis}
given by
\begin{equation}\label{OU}
X(t,x)=e^{tA}x+ \int_0^t e^{(t-s)A}\sigma dW(s), \ \ \ t\geq 0.
\end{equation}
For the study of the properties of this process one can see \cite[Ch.\,5 and 9]{DaPratoZabczyk14}.
Let us denote  by $R_t$ the associated one parameter transition semigroup, i.e.
\begin{equation}\label{eq4:defRt}
R_{t}[\phi](x)
:=\E\left[\phi(X(t,x))\right]=
\E\left[\phi(e^{tA}x+W_A(t))\right], \ \ \phi\in B_m(H),
\end{equation}
where
$$W_A(t)=\int_0^t \ets \sigma dW(s).$$
Under Assumption \ref{hp4:ABQforOU}, one can define the following operator as Bochner integral in the separable Hilbert space $\call_1(H)$:
\begin{equation}\label{eq4:defQtnew}
    Q_t:H \to H, \qquad
    Q_{t}:=\int_{0}^{t}e^{sA}\sigma\sigma^{\ast}e^{sA^{\ast}}ds,
\end{equation}
Clearly $Q_t$ is nonnegative, i.e. $Q_t\in \call_1^+(H)$  for every $t\geq0$.
 It turns out that the law of $W_A(t)$  is $\caln_{Q_t}$, see \cite[Th.\,5.2]{DaPratoZabczyk14}.
Then,
 we can write
\begin{equation}\label{eq4:defRtbis}
R_{t}[\phi](x)
= \int_{H}\phi(  e^{tA}x+  y)
\mathcal{N}_{Q_{t}}(dy).
\end{equation}
As we have observed, as a general consequence of Theorem \ref{th1:exSDEmult}, Assumptions \ref{hp4:Psemigroupell}(i)--(iii) (hence,  Assumption \ref{hp4:Psemigroupellbis}(i)) are automatically satisfied by $(R_t)_{t\ge 0}$.
 We are going to provide conditions
that guarantee the last three properties of Assumption  \ref{hp4:Psemigroupellbis}. Before, we provide the following example, serving as a key motivation for our approach (see the discussion at the beginning of Section \ref{sec:sem}).
\begin{Example} (see also \cite[Ex.\,6.1]{Cerrai94}) \label{OU-notmes}
Let $H=\R$, $A,\sigma\in \R$ with $A\neq 0$, and consider the semigroup $(R_t)_{t\geq 0}$ in the space $C_b(\R)$ associated to $A,\sigma$. We claim (and prove) the following.
\begin{enumerate}[(i)]
\item Let $\varphi(x)\coloneqq \sin(x)$. The map $[0,+\infty) \to C_b(H)$, $t\mapsto  R_t[\varphi]$ is not continuous at any $t\geq 0$.
\item There exists $\phi\in C_b(\R)$ such that the map $[0,T]\to C_b(H)$, $t\mapsto R_t\phi$ is not measurable (in the sense of Bochner integration, i.e. it is not the limit of finite valued functions) for any $T>0$.
\end{enumerate}

\emph{Proof of (i).}
%
Let $i$ denote the imaginary unit. For every $t,s\geq 0$ we can write
\begin{eqnarray*}
\int_{\R} e^{i(e^{At}x+y)} N_{Q_t}(dy)-\int_{\R} e^{i(e^{As}x+y)} N_{Q_t}(dy) &=& e^{ie^{At}x} \left(\int_{\R} e^{iy}\caln_{Q_t}(dy)-\int_{\R} e^{iy}\caln_{Q_s}(dy)\right)\\
&&+ \left(e^{ie^{At}x}- e^{ie^{As}x}\right) \int_{\R} e^{iy}\caln_{Q_s}(dy)\\
&=& e^{ie^{At} x} \left(e^{-\frac{1}{2}Q_t^2}-e^{-\frac{1}{2}Q_s^2}\right)+
 \left(e^{ie^{At}x}- e^{i^{As}x}\right) e^{-\frac{1}{2}Q_s^2}.
\end{eqnarray*}
Performing the same computation with $-i$ in place of $i$ and using the identity $\sin \alpha =\frac{e^{i\alpha}-e^{-i\alpha}}{2i}$, we get
$$
R_t[\varphi](x)-R_s[\varphi](x)= \sin (e^{At}x) \left(e^{-\frac{1}{2}Q_t^2}-e^{-\frac{1}{2}Q_s^2}\right)+
\left(\sin (e^{At}x)- \sin (e^{As}x)\right) e^{-\frac{1}{2}Q_s^2}.
$$
Now take $t\to s$ above. The first addend in the right hand side goes to $0$ uniformly in $x\in \R$, whereas the second addend does not do so. Hence, the claim follows.

\emph{Proof of (ii).}
We know that for every $T>0$ there exists $M\geq 1$ such that $|R_t|_{\call(C_b(\R))}\leq M$ for every $t\in [0,T]$. Assume, by contradiction, that the function $[0,T]\to C_b(\R)$, $t\mapsto R_t[\phi]$ is measurable for every $\phi\in C_b(\R)$. Then
\begin{enumerate}
\item  the semigroup $(R_t)_{t\geq 0}$ is weakly measurable, i.e. the map $[0,+\infty)\to\R$, $t\mapsto \langle R_t\phi,\phi^*\rangle_{C_b(\R),C_b(\R)^*}$ is measurable for each couple $(\phi,\phi^*)\in   C_b(\R)\times C_b(\R)^*$;
\item the set $\{R_t[\phi],\ t\in [0,T]\}$ is separable in $C_b(\R)$ for each $\phi\in C_b(\R)$, as $R_\cdot[\phi]$ is the limit of functions with finite values in $C_b(\R)$.
\end{enumerate}
Then, applying   \cite{Dunford38}, we contradict item (i).
\end{Example}

 Throughout the rest of this subsection,  we consider  maps $G:H\to \call_u(K,H)$ that are constant ($G(\cdot)\equiv G_0\in \call_u(K,H)$). Clearly these maps satisfy Assumptions \ref{assG} and \ref{assGbis}. With a slight abuse of notation, we will confuse $G$ and $G_0$.
We  consider  the following assumption, introduced first in \cite{Zabczyk85} in the case $G=I$, which guarantees the $G$-differentiability of $x\mapsto R_t[\phi](x)$.
\begin{Assumption}
\label{hp4:NCG} Let  $G:\cald(G)	\subseteq K \to H$ be a closed densely defined (possibly unbounded) linear operator.
\begin{enumerate}[(i)]
\item
If $G$ is bounded, we assume
\begin{equation}
e^{tA}G(K)\subseteq Q_{t}^{1/2}(H), \qquad \forall t > 0.
\label{eq4:ornsteininclusionebdd}
\end{equation}
\item
If $G$ is unbounded, we assume that for all $t>0$ the operator $e^{tA}G:\cald(G)\subseteq K\to H$ can be extended to
a  $\overline{e^{tA}G}\in\call(K,H)$ such that
\begin{equation}
\overline{e^{tA}G}(K)\subseteq Q_{t}^{1/2}(H), \qquad \forall t > 0.
\label{eq4:ornsteininclusioneunbdd}
\end{equation}
\end{enumerate}
\end{Assumption}

\begin{Remark}
\label{rm4:NC1}
When $K=H$ and $G=I$, Assumption \ref{hp4:NCG} is equivalent to require that the deterministic control system in $H$
 \begin{equation}
\label{eq4:detsystNC}
z'(s)=Az(s) +\sigma u(s),\quad z (0)=z_0\in H.
\end{equation}
is {\em null controllable} from every initial datum $z_0 \in  H$; that is,  for every $t>0$ and $z_0 \in H$,
there exists a control $u(\cdot)\in L^2([0,t],\Xi)$ such that $z(t;z_0,u(\cdot))=0$.
In terms of operators, as
$$
z(t)=e^{tA}x +\int_0^t e^{(t-s)A}\sigma u(s)ds,
$$
denoting by $\call_t $ the operator
$$
\call_t:L^2([0,t],\Xi)\to H, \qquad \call_t u\coloneqq\int_0^t e^{(t-s)A}\sigma u(s)ds,
$$
the null controllability for an initial datum $z_0 \in  H$ corresponds to
$e^{tA}z_0 \in \call_t(L^2([0,t],\Xi)).$
Then , the equivalence aforementioned  follows  from the fact that,
after some computations, one gets
$$
|\call^*_t x|^2=\<Q_t x,x\>=|Q_t^{1/2} x|^2, \quad x \in H,
$$
therefore (see \cite[Cor.\,B.7]{DaPratoZabczyk14})
\begin{equation}\label{eq4:imQt=imLt}
\call_t(L^2[0,t],\Xi))=Q_t^{1/2}(H).
\end{equation}
When $G$ is bounded, in view of what we said above, Assumption \ref{hp4:NCG} is equivalent to ask that system \myref{eq4:detsystNC} is null controllable for every initial datum $z_0 \in G(K)\subseteq H$ (see also  \cite[Sec.\,3.1]{Masiero05}).

When $G$ is unbounded, we may consider Assumption \ref{hp4:NCG} as a null controllability assumption
for the extension of system \myref{eq4:detsystNC} to a suitable extrapolation
space(see e.g. \cite[Sec.\,II.5]{EngelNagel99}).

Finally, when $G=I$, if \myref{eq4:ornsteininclusionebdd}
holds for a given $t_0>0$, it must hold for all $t>t_0$, as $e^{tA}{{(H)}}$ does not increase in $t$ (by the semigroup property), whereas
$Q_t^{1/2}(H)$ does not decrease  (by \myref{eq4:imQt=imLt}\footnote{It is indeed constant when $G=I$ and \myref{eq4:ornsteininclusionebdd} holds, see
e.g. \cite[Th.\,B.2.2]{DaPratoZabczyk02}}).
This is not ensured when $G\not=I$.
\end{Remark}\vspace{3mm}

If Assumption \ref{hp4:NCG}  holds, the operator
\begin{equation}\label{eq4:defGammaG}
\Gamma_G(t):K\to H, \  \  \ \Gamma_G(t)\coloneqq Q_{t}^{-1/2}e^{tA}G
\qquad \hbox{(or $Q_{t}^{-1/2}\overline{e^{tA}G}$ when $G$ is unbounded)},
\end{equation}
where $Q_{t}^{-1/2}$ is the pseudoinverse of $Q_{t}^{1/2}$, is well defined for all $t>0$. Moreover, it is possible to check that it is bounded by the closed graph theorem, so it belongs to $\call(K,H)$.
When $K=H$ and $G=I$, we simply write $\Gamma(t):= Q_{t}^{-1/2}e^{tA}$.

\smallskip

We begin recalling a classical result concerning the case when  Assumption \ref{hp4:NCG} holds with $K=H$ and  $G=I$ (see
  \cite[Th.\,6.2.2 and Ex.\,6.3.3]{DaPratoZabczyk02}).

\begin{Theorem}
\label{th4:regOU}
Let Assumptions \ref{hp4:ABQforOU} and \ref{hp4:NCG} hold with $K=H$ and $G=I$. Then,  for each $\phi \in  B _b (H)$ and   $t > 0$, we have
$R_t [\phi] \in UC^{\infty }_b (H)$(\footnote{$UC^{\infty }_b (H)$ denotes the space of uniformly continuous and bounded functions on $H$ having Fr\'echet derivatives of each order, all uniformly continuous and bounded as well.})
In particular, for each $k,h,x\in H$ and $t>0$, we have
\begin{equation}
\label{e6.3.8}
\langle D R_t [\phi ](x), k \rangle=\int_H \langle \Gamma(t) k,
Q^{-1/2}_t y \rangle \phi (e^{tA}x + y) \mathcal{N}_{Q_{t}}(dy),
\end{equation}
and
\begin{equation}
|DR_{t}[\phi] (x)|\le |\Gamma(t)|_{\call(H)}\;|\phi |_{0}.
\end{equation}

Conversely, if  Assumption \ref{hp4:ABQforOU} holds with $K=H$ and $G=I$ and
$R_{t}[\phi ]\in C_{b}(H)$ for each $\phi \in  B _{b}(H)$
and $t>0$,  then  \myref{eq4:ornsteininclusionebdd} is satisfied (with $K=H$ and $G=I$).
\end{Theorem}


We are going to prove an analogous result for the case of $G$-derivatives
(with $G$ possibly unbounded),  generalizing the result of \cite[Lemma\,3.4]{Masiero05}. The latter result can be also found in
a slightly more general form in \cite{GozziMasiero12}.
First, we need two lemmas. In the following, the symbol  $[t]$ denotes the integer part of $t\in[0,+\infty)$.

\begin{Lemma}
\label{lm4:th4regOUG1} Let Assumption \ref{hp4:ABQforOU} hold and let $M\geq 1$, $\omega\in \R$ be as in \eqref{stimasem}.
Then
$$
{\rm Tr} [Q_t] \le {\rm Tr}[Q_{1}] M^2\frac{e^{2\omega ([t]+1)}-1}{e^{2\omega}-1},  \ \ \ \forall t\geq 0,
$$
with the agreement
$$\frac{e^{2\omega ([t]+1)}-1}{e^{2\omega}-1}\coloneqq [t]+1, \ \ \mbox{if} \ \omega=0.$$

\end{Lemma}

\begin{proof}
%
Note that
$$
Q_t=Q_{t-1}+\int_{t-1}^{t}e^{sA}\sigma\sigma^*e^{sA^*}ds=Q_{t-1}
+e^{(t-1)A}Q_{1}e^{(t-1)A^*}, \ \ \ \forall t\geq 1,$$
Now, recall that  if $T\in\call_1(H)$ and $S\in \call(H)$, then $TS\in\call_1(H)$ and $|TS|_{\call_1(H)}\leq |T|_{\call_1(H)}|S|_{\call(H)}$ and that the trace is additive.  Then, setting $a_n\coloneqq \mbox{Tr}[Q_n]$, $n\in\N$, and $q\coloneqq\mbox{Tr}[Q_1]$, we get
$$a_0=0, \ \ a_n\leq a_{n-1}+qM^2e^{2\omega(n-1)} \ \forall n\in\N\setminus\{0\}.$$
Then,
$
a_n\leq qM^2 \sum_{k=1}^n e^{2\omega(k-1)}=qM^2\frac{e^{2\omega n}-1}{e^{2\omega}-1}
$ (with the agreement specified in the statement when $\omega=0$).
The claim follows simply observing that $t \le [t]+1$.
%
So, the claim  follows from
the fact that
 $t\mapsto {\rm Tr} [Q_t]$
is clearly increasing in $t$.
\end{proof}

The following result is an extension of Proposition 2.19 of
\cite{DaPratoZabczyk02} and of Lemma 3.1 of \cite{Cerrai95}.

\begin{Lemma}
\label{lm4:th4regOUG2}
Let Assumption \ref{hp4:ABQforOU} hold and let $M\geq 1$, $\omega\in \R$ be as in \eqref{stimasem}.
Then, for every $m\geq 0$, there exists a constant $\kappa>0$ such that
\begin{equation}
\int_H|y+e^{tA}x|^m\caln_{Q_{t}}(dy)
\le
\kappa (1+|x|^m) e^{m\omega t}, \ \ \ \forall t\geq 0, \ \forall x\in H, \qquad \mbox{if} \; \omega \ne 0,
\end{equation}
\begin{equation}\label{eq4:estimateintOUbis}
\int_H
|y+e^{tA}x|^m\caln_{Q_{t}}(dy)
\le
\kappa(1+|x|^m)(1+t^m), \ \ \ \forall t\geq 0, \ \forall x\in H,  \qquad \mbox{if} \; \omega = 0.
\end{equation}
\end{Lemma}

\begin{proof}
The case $m=0$ is obvious. Let $m> 0$.
We have, for every $t\geq 0$, $x,y\in H$,
$$
\int_H |y+e^{tA}x|^m\caln_{Q_{t}}(dy)
\le (1\vee 2^{m-1})
\left(|e^{tA}x|^{m}+ \int_H |y|^{m}\caln_{Q_{t}}(dy)\right).
$$
By \cite[Prop.\,2.19, p.\,50]{DaPratoZabczyk14}\footnote{The claim, stated for $m/2\in\N$, may be extended to each $m\geq 0$ by taking into account that, for $k-1<m/2<k$, $k=0,1,...,$ we can write
$$
\int_H |y|^{m} \caln_{Q_{t}}(dy) \le
\left[\int_H |y|^{2k} \caln_{Q_{t}}(dy)\right]^{m/2k}.
$$}, we have
for some $\kappa'>0$ independent of $t$,
$$
\int_H |y|^{m} \caln_{Q_{t}}(dy) \le \kappa'({\rm Tr}  [Q_t])^{m/2}.
$$
Then, we use  Lemma \ref{lm4:th4regOUG1} and \eqref{stimasem} to  obtain, for every  $t\geq 0$ and $x\in H$,
$$
\int_H|y+e^{tA}x|^m\caln_{Q_{t}}(dy)
\le (1\vee 2^{m-1})
\left(M^{m} e^{m\omega t}|x|^{m}+ \kappa'({\rm Tr}[Q_{1}])^{m/2} M^m\left[\frac{e^{2\omega ([t]+1)}-1}{e^{2\omega}-1}\right]^{m/2}\right),
$$
with the agreement
$$\frac{e^{2\omega ([t]+1)}-1}{e^{2\omega}-1}\coloneqq t+1 \ \ \mbox{if} \ \omega=0,$$
and the claim follows by defining a suitable constant $\kappa>0$.
\end{proof}

\begin{Theorem}
\label{th4:regOUG} Let Assumptions \ref{hp4:ABQforOU} and \ref{hp4:NCG}
hold true, and let $M\geq 1$, $\omega\in \R$ be as in \eqref{stimasem}. Then, for every $m\geq 0$,  we have the following statements.
\begin{enumerate}[(i)]
  \item
For every $\phi \in B_{m}(H)$ and $t>0$, the function
$R_t[\phi]:H\to \R$ is $G$-Fr\'echet differentiable and
\begin{equation}\label{eq4:derOUG}
\< D^G R_{t} [\phi] (x),k\>_K=\int_{H}\phi
\left(  y+e^{tA}x\right)  \left\langle \Gamma_G(t)k,Q_{t}^{-1/2}y\right\rangle_H
 \mathcal{N}_{Q_{t}}(dy), \ \ \ \ \forall x\in H,\ \forall k \in K.
\end{equation}
\item
There exists a constant $\kappa_G>0$ such that
\begin{equation}\label{eq4:derOUGstimader}
|D^G R_{t} [\phi] (x)|_K \le \kappa_G|\Gamma_G(t)|_{\call(K,H)} (1+|x|^m) e^{m\omega t} |\phi|_{ B _m(H)}, \ \ \  \forall t>0, \ \forall x\in H, \ \ \ \mbox{if} \ \omega> 0,
\end{equation}
\vspace{-15pt}
\begin{equation}\label{eq4:derOUGstimaderbis}
|D^G R_{t} [\phi] (x)|_K \le \kappa_G|\Gamma_G(t)|_{\call(K,H)} (1+|x|^m) (1+t^{m/2}) |\phi|_{ B _m(H)},\ \ \  \forall t>0, \ \forall x\in H, \ \ \ \mbox{if} \ \omega=0,
\end{equation}
\begin{equation}\label{eq4:derOUGstimadertris}
|D^G R_{t} [\phi] (x)|_K \le \kappa_G|\Gamma_G(t)|_{\call(K,H)} (1+|x|^m)  |\phi|_{ B _m(H)}, \ \ \  \forall t>0, \ \forall x\in H, \ \ \ \mbox{if} \ \omega< 0.
\end{equation}
  \item
If $\phi \in C_m(H)$, then also $D^G R_{t} [\phi]\in C_m(H,K)$
for all $t>0$.
  \item
If $\phi \in C^1_m(H)$, then
\begin{equation}\label{eq4:derOUGdiff}
\< D^G R_{t} [\phi] (x),k\>_K=\int_{H}\<D\phi
\left(  y+e^{tA}x\right),\overline{e^{tA}G}k \>_H \mathcal{N}_{Q_{t}}(dy), \ \ \ \forall t>0.
\end{equation}
\end{enumerate}
\end{Theorem}

\begin{proof}

{\em (i)}
We  compute, for $t>0$, $x\in H$, $k\in \cald(G)\subseteq K$, the limit
$$
\lim_{s\rightarrow0}\frac{1}{s}\left[ R_{t} [\phi] (x+sGk)- R_{t} [\phi] (x)\right].
$$
Using \myref{eq4:defRtbis} we have for every $s\neq 0$
\begin{align*}
&  \frac{1}{s}\left[ R_{t} [\phi] (x+sGk)- R_{t} [\phi] (x)\right] \\
&  =
\frac{1}{s}\left[\int_{H}\phi\left(  y+e^{tA}(x+sGk)\right)
\mathcal{N}_{Q_{t}}(dy)  -\int_{H}
\phi\left(  y+e^{tA}x\right)  \mathcal{N}_{Q_{t}}(dy)\right] \\
&  =
\frac{1}{s}\left[\int_{H}\phi\left(  y+e^{tA}x\right)
\mathcal{N}\left( se^{tA}Gk ,Q_{t}\right)  \left(  dy\right)  -\int_{H}
\phi\left(  y+e^{tA}x\right)  \mathcal{N}_{Q_{t}}(dy) \right].
\end{align*}
By Remark \ref{rem:inverse},
$\calr(Q_{t}^{-1/2})=\left[\ker\left(  Q_{t}^{1/2}\right)\right]^{\perp}$.
Hence, as $Q^{1/2}_t$ is self-adjoint, $\Gamma_G(t)k\in\overline{\calr(Q_{t}^{1/2})}$.
So,
by \cite[Sec.\,1.2.4]{DaPratoZabczyk02},
the  map
$\calr(Q_t^{1/2}) \to \R$, $ y \mapsto \left\langle \Gamma_G(t)k,Q_{t}^{-1/2}y\right\rangle $ extends to
a linear continuous functional on $H$ (still denoted by the same expression, with a slight abuse of notation), which is square integrable  with respect to the measure
$\mathcal{N}_{Q_{t}}$.

Now,
by Assumption \ref{hp4:NCG},
$se^{tA}Gk\in Q_{t}^{1/2}(H)$.
Hence, by  \cite[Th.\,1.3.6]{DaPratoZabczyk02}, the Gaussian measures $\mathcal{N}\left(  se^{tA}Gk,Q_{t}\right) $ and
$\mathcal{N}_{Q_{t}}$ are equivalent and we can apply the Cameron-Martin formula:
$$
d\left(s,y;k\right)     \coloneqq\frac{d\mathcal{N}\left(  se^{tA}Gk,Q_{t}
\right)  }{d\mathcal{N}_{Q_{t}}}(y)
 =\exp\left\{  \left\langle s\Gamma_G(t)k,Q_{t}^{-1/2}y\right\rangle_H
-\frac{1}{2}s^{2}\left| \Gamma_G(t)k\right|_H  ^{2}\right\}, \ \ \forall s\in \R, \ \forall y\in H.
$$
Notice that
\begin{equation}\label{uniflim}
\lim_{s\rightarrow0}\frac{
d\left(  s,y;k\right)  -1  }{s} =\left\langle \Gamma_G(t)k,Q_{t}^{-1/2}y\right\rangle_H \ \ \ \mbox{uniformly in } k\in\cald(G)\cap B_K(0,1),
\end{equation}
and $\frac{d\left(  s,\cdot;k\right)  -1  }{s}\leq \alpha(\cdot,k)\in B_b(H)$ for every $s\in[-1,1]\setminus\{0\}$.
So
\begin{align*}
& \lim_{s\rightarrow0}\frac{1}{s}\left[ R_{t} [\phi] (x+sGk)- R_{t} [\phi] (x)\right]
  =\lim_{s\rightarrow0}\int_{H}\phi\left(  y+e^{tA}x\right)  \frac{
d\left(  s,y;k\right)  -1 }{s}\mathcal{N}_{Q_{t}}(dy) \\
&  =\int_{H}\phi\left(  y+e^{tA}x\right)  \lim_{s\rightarrow0}\frac{
d\left(  s,y;k\right)  -1  }{s}\mathcal{N}_{Q_{t}}(dy)  =\int_{H}\phi\left(  y+e^{tA}x\right)  \left\langle \Gamma_G(t)k,Q_{t}^{-1/2}y\right\rangle_H \mathcal{N}_{Q_{t}}(dy),
\end{align*}
i.e., there exists the $G$-directional derivative $\nabla^GR_t[\phi](x;k)$ and
$$
\nabla^GR_t[\phi](x;k)=\int_{H}\phi\left(  y+e^{tA}x\right)  \left\langle \Gamma_G(t)k,Q_{t}^{-1/2}y\right\rangle_H \mathcal{N}_{Q_{t}}(dy).
$$
Using the Cauchy-Schwarz inequality,
\begin{align}
&  \left|  \int_{H}\phi\left(  y+e^{tA}x\right)
\left\langle \Gamma_G(t)k,Q_{t}^{-1/2}y\right\rangle
\mathcal{N}_{Q_{t}}(dy)  \right|
\nonumber
\\
& \leq \left|  \phi \right|_{ B _m(H)}  \int_{H}(1+|y+e^{tA}x|^{m})
\left|\left\langle \Gamma_G(t)k,Q_{t}^{-1/2}y\right\rangle \right|
\mathcal{N}_{Q_{t}}(dy)
\nonumber
\\
&  \leq\left|  \phi \right|_{ B _m(H)}
 \left(  \int_{H}
(1+|y+e^{tA}x|^{m})^2
\mathcal{N}_{Q_{t}}(dy)  \right)^{1/2}
\left(  \int_{H}\left|  \left\langle \Gamma_G(t)k,Q_{t}^{-1/2}y\right\rangle
\right|  ^{2}\mathcal{N}_{Q_{t}}(dy) \right)^{1/2}.
\nonumber
\label{eq4:estimateNCG}
\end{align}
Then, using Lemma \ref{lm4:th4regOUG2} and the fact that
$$\int_H\langle y,\xi\rangle \langle y,\eta\rangle\caln_{Q_t}(dy)=\langle Q_t\xi,\eta\rangle, \ \ \ \forall \xi,\eta\in H,
$$
 we get, for some $\kappa_G>0$,
 \begin{equation}\label{mmm}
\left| \nabla^GR_t[\phi](x;k) \right|_K   \le \kappa_G\left|  \phi \right|_{ B _m(H)}
e^{m\omega t}(1+|x|_H^m)  \left|
\Gamma_G(t)\right|_{\call(K,H)}  \left|  k\right|_K, \ \ \ \mbox{if}  \ \omega> 0,
\end{equation}
  \begin{equation}\label{mmm2}
\left| \nabla^GR_t[\phi](x;k) \right|_K   \le \kappa_G\left|  \phi \right|_{ B _m(H)}
(1+t^m)(1+|x|_H^m)  \left|
\Gamma_G(t)\right|_{\call(K,H)}  \left|  k\right|_K, \ \ \ \mbox{if}  \ \omega= 0,
\end{equation}
\begin{equation}\label{mm3}
\left| \nabla^GR_t[\phi](x;k) \right|_K \le \kappa_G|\Gamma_G(t)|_{\call(K,H)} (1+|x|^m) |\phi|_{ B _m(H)}, \ \ \  \forall t>0, \ \forall x\in H, \ \ \ \mbox{if} \ \omega< 0,
\end{equation}
This shows that
the linear functional on $(\cald(G),|\cdot|_K)$ defined by $k\mapsto  \nabla^G R_t[\phi](x;k)$ can be extended to a continuous linear functional on $K$, that is,  $R_t[\phi]$ is $G$-Gateaux differentiable at $x$ with $G$-Gateaux derivative $\nabla^GR_t[\phi](x)\in\call(K,H)$ and
\begin{equation}\label{ppppkbis}
\langle \nabla^GR_t[\phi](x), k\rangle_K = \int_{H}\phi\left(  y+e^{tA}x\right)  \left\langle \Gamma_G(t)k,Q_{t}^{-1/2}y\right\rangle_H \mathcal{N}_{Q_{t}}(dy), \ \ \ \forall k\in \cald(G).
\end{equation}
On the other hand, taking into account that the limit in \eqref{uniflim} is uniform in $k\in\cald(G)$, we conclude that $R_t[\phi]$ is actually $G$-Fr\'echet differentiable at $x$ and \eqref{eq4:derOUG} follows from \eqref{ppppk}.
\medskip

\emph{(ii)} This follows from (i) and from \eqref{mmm}--\eqref{mm3}.
\medskip

%
%

{\em (iii)}
Let now $\phi \in C_m(H)$, $x\in H$,
and take a sequence $x_n \to x$ in $H$.
Then, by (i)  we get
\begin{align*}
|D^{G}R_{t}[\phi](x_n)-D^{G}R_{t}[\phi](x)|_K&=
\sup_{|k|_K=1}\<D^{G}R_{t}[\phi](x_n)-D^{G}R_{t}[\phi](x),k\>_K
\\&=
\sup_{|k|_K=1}\int_{H}\left[\phi\left(y+{e^{tA}}x_n\right)-\phi\left(y+{e^{tA}}x\right) \right] \left\langle \Gamma_G(t)k,Q_{t}^{-1/2}y\right\rangle_H
\mathcal{N}_{Q_{t}}(dy)
\\
&
\le
\left(\int_{H}\left|\phi\left(y+{e^{tA}}x_n\right)-\phi\left(y+{e^{tA}}x\right)
\right|^2\mathcal{N}_{Q_{t}}(dy)
\right)^{1/2}
|\Gamma_G(t)|_{\call(K,H)}.
\end{align*}
Hence the claim follows by the dominated convergence theorem.

\medskip

{\em (iii)}
Let $\phi \in C^1_m(H)$, $t>0$, and $k\in D(G)\subseteq K$. We have, by the dominated convergence theorem,
\begin{align*}
\< D^G R_{t} [\phi] (x),k\>_K&=\lim_{s\rightarrow0}\frac{1}{s}\left[ R_{t} [\phi] (x+sGk)- R_{t} [\phi] (x)\right]\\
&  =
\int_{H}\lim_{s\rightarrow0}\frac{1}{s} \left[\phi\left(  y+e^{tA}(x+sGk)\right)-\phi\left(  y+e^{tA}x\right)\right]
\mathcal{N}_{Q_{t}}(dy)
\\
&  =
\int_{H}\<D\phi\left(  y+e^{tA}x\right),e^{tA}Gk\>_H  \mathcal{N}_{Q_{t}}(dy)  .
\end{align*}
The claim when $k \in K$ simply follows
using the density of $\cald(G)$ in $K$ and the fact that, by Assumption \ref{hp4:NCG},
the operator $e^{tA}G$ can be extended  $\overline{e^{tA}G}\in\call(K,H)$.
\end{proof}
\begin{Corollary}\label{corcor}
Let Assumptions \ref{hp4:ABQforOU} and \ref{hp4:NCG}
hold true. Assume, moreover, that $t\mapsto |\Gamma_G(t)|_{\call(K,H)}$ belongs to $\cali$. Then the family of linear operators $(R_t)_{t\geq 0}$ satisfies Assumption \ref{hp4:Psemigroupellbis} for every $m\geq 0$.
\end{Corollary}
\begin{proof}
As we have already observed, the conditions of Assumption \ref{hp4:Psemigroupell}(i)--(iii) (i.e Assumption \ref{hp4:Psemigroupellbis}(i))  are a straightforward consequence of the general Theorem \ref{th1:exSDEmult}.

Combining the assumption that $t\mapsto |\Gamma_G(t)|_{\call(K,H)}$ belongs to $\cali$ and Theorem \ref{th4:regOUG},  the conditions of Assumptions \ref{hp4:Psemigroupellbis}
(\ref{hp4:smoothingGellbis}) and (\ref{ip4bis}) follow.

Finally, the condition of Assumption \ref{hp4:Psemigroupellbis}(\ref{ip5bis}), i.e. the fact that the map
 $(0,+\infty)\times H \to K,$ $(s,x) \mapsto D^G R_{s}[\phi](x)$ is measurable, follows from the representation formula \eqref{eq4:derOUG} and using Pettis measurability's Theorem \cite[Th.\,1.1]{Pettis38}, as $K$ is separable.
\end{proof}

\begin{Example}
\label{ex1}
(cf. \cite[Sec.\,9.6]{DaPratoZabczyk14},
\cite[Ex.\,6.2.11] {DaPratoZabczyk02},
 \cite[Sec.\,6]{Gozzi95}, \cite[Ex.\,3.8]{Masiero05}).
Let $K=\Xi=H$, $\sigma\in \call(\Xi,H)$ and $Q :=\sigma\sigma^*$.
Consider an orthonormal basis $\{e_{n}\}_{n\in\N}$ in $H$
and assume that $A$, $Q$, and $G$ admit spectral
decompositions
$$
Ae_{n}=-\alpha_{n}e_{n}, \qquad Q e_{n}=q_{n}e_{n},
\qquad Ge_{n}=g_{n}e_{n}, \ \ \  \ \ \ \forall n\in\mathbb{N},
$$
where
$\alpha_{n}\geq 0$, $g_n\in \R$,  $q_{n}>0$ for all $n\in\N$ and $\alpha_{n}\uparrow+\infty$ as $n\to\infty$.
Then
$$
e^{sA}Qe^{sA^*}e_n=e^{-2s\alpha_n}q_ne_n, \ \ \ n\in\N, \ s>0.
$$
Note that,   the set $\Sigma_0\coloneqq \{n\in\N: \ \alpha_n=0\}$ is finite\,\footnote{In the quoted
references  at the beginning of this example it is assumed
$\alpha_n>0$. Here we let some $\alpha_n$ to be $0$ to cover the case
 developed in Subsection 5.1.}. Call $\Sigma_0^c=\N\setminus \Sigma_0$.
Then
Assumption \ref{hp4:ABQforOU}(ii) holds true {{if
and only if  $$\int_0^t\sum_{n\in \Sigma_0^c}q_ne^{-2s\alpha_n} ds < + \infty  \ \ \ \forall t\geq 0;$$  by Fubini-Tonelli's Theorem and arbitrariness of $t\geq 0$, this happens}} if and only if
\begin{equation}\label{eq4:Qtdiagnuclear}
\sum_{n\in \Sigma_0^c}\frac{q_n}{2\alpha_n}<+\infty.
\end{equation}
In this case
$Q_t$ is diagonal too and
$$
Q_t e_n=\frac{q_n}{2\alpha_n} (1-e^{-2\alpha_n t})e_n, \ \ \ \forall n\in\N,
$$
under the agreement $\frac{1-e^{-2\alpha_nt}}{2\alpha_n}\coloneqq t$ if $n\in \Sigma_0$.
On the other hand, the operator $e^{tA}G$ extends to a bounded operator
$\overline{e^{tA}G}\in\call(H)$ for every $t>0$
if and only if
\begin{equation}\label{eq4:etAGextend}
\sup_n \left| e^{-t\alpha_{n}}g_n \right|< + \infty, \ \ \ \forall t>0.
\end{equation}
Hence, formally we have, with the usual  agreement $\frac{1-e^{-2\alpha_nt}}{2\alpha_n}\coloneqq t$ if $n\in \Sigma_0$,
$$
\Gamma_G(t)e_n= Q_t^{-1/2}\overline{e^{tA}G} e_n
=
\sqrt{
\frac{2\alpha_{n}}{(1-e^{-2t\alpha_{n}})q_n}
}\;\;
e^{-t\alpha_{n}}g_n e_n
\ \ \ \forall n\in\N.
$$
This shows that Assumption \ref{hp4:NCG} holds if and only if
(here with the agreement that
$\frac{2\alpha_{n}}{e^{2t\alpha_{n}}-1}
\coloneqq t^{-1}$ if $n\in \Sigma_0$),
\begin{equation}\label{eq4:diagNC}
\sup_{n\in\mathbb{N}}\sqrt{
\frac{2\alpha_{n}}{e^{2t\alpha_{n}}-1}
\cdot \frac{g_n^2}{q_n} }\;\; <+ \infty, \ \ \ \forall t>0.
\end{equation}
Such supremum is indeed equal to $|\Gamma_G(t)|_{\call(H)}$, so
 $t\mapsto |\Gamma_G(t)|_{\call(K,H)}$ belongs to $\cali$ if and only if
there exists a function $\gamma_G\in\cali$ such that
\begin{equation}\label{eq4:diagGammaGint}
\sup_{n\in\mathbb{N}}\sqrt{\frac{2\alpha_{n}}{e^{2t\alpha_{n}}-1}
\cdot \frac{g^2_n}{q_n}} \le \gamma_G (t), \qquad \forall t >0.
\end{equation}
Note that \eqref{eq4:diagGammaGint} implies, in particular, \eqref{eq4:diagNC}.
Therefore, if all the conditions listed above are fulfilled, Corollary \ref{corcor} applies.

We now discuss more in detail three particularly  meaningful cases.
\begin{itemize}
\item[(1)]
$G=I$, $-A>0$, $Q=(-A)^{-\beta}$ for some  $\beta \in [0,1)$. Clearly, \eqref{eq4:etAGextend} is satisfied. Moreover,
$$
|\Gamma_G(t)|_{\call(H)}^2=
\sup_{n\in\mathbb{N}}
\frac{2\alpha_n^{1+\beta}}{(e^{2t\alpha_{n}}-1)}\le \frac{C_0}{t}, \ \ \ \forall t>0,
$$
where $C_0:= \sup_{s>0}\frac{2s^{1+\beta}}{e^{2s}-1}$, so
\myref{eq4:diagGammaGint} holds with $\gamma_G(t)=C_0^{1/2}t^{	\beta}$.
Finally, in this case, \myref{eq4:Qtdiagnuclear} is verified if
$$
\sum_{n \in \N}\alpha_n^{-1-\beta}< + \infty,
$$
which is true, e.g.,  if $\alpha_n \sim n^{\eta}$
for $\eta> \frac{1}{1+\beta}$.

\item[(2)] $G=\sqrt{Q}$ (special subcase: $G=Q =I$).
Clearly, \eqref{eq4:etAGextend} is satisfied. Moreover,  with the usual  agreement $\frac{1-e^{-2\alpha_nt}}{2\alpha_n}\coloneqq t$ if $n\in \Sigma_0$,
$$
|\Gamma_G(t)|_{\call(H)}^2=
\sup_{n\in\mathbb{N}}
\frac{2\alpha_{n}}{e^{2t\alpha_{n}}-1}\le \frac{C_0}{t}, \ \ \ \forall t>0,
$$
where $C_0\coloneqq \sup_{s>0} \frac{s}{e^s-1}$, so
\myref{eq4:diagGammaGint} holds with $\gamma_G(t)=C_0^{1/2}t^{-1/2}$.

\item[(3)]
 $G=(-A)^\beta\sqrt Q$ for some  $\beta \in [0,1/2)$.
In this case we have $g_n=\alpha_n^{\beta}\sqrt {q_n}$ for every $n \in \N$. Hence, \eqref{eq4:etAGextend} is satisfied. Moreover,  with the usual  agreement $\frac{1-e^{-2\alpha_nt}}{2\alpha_n}\coloneqq t$ if $n\in \Sigma_0$,
\[
|\Gamma_G(t)|_{\call(H)}^2=
\sup_{n\in\mathbb{N}}
\frac{2\alpha_{n}^{1+2\beta}}{e^{2t\alpha_{n}}-1}\le \frac{C_0}{t^{1+2\beta}},
\]
where $C_0\coloneqq \sup_{s>0} \frac{s^{1+2\beta}}{e^s-1}$,  so \myref{eq4:diagGammaGint} holds with $\gamma_G(t)=\left( \frac{C_0}{t^{1+2\beta}}\right)^{1/2}$.
\end{itemize}
\end{Example}

\subsection{The case of invertible diffusion coefficient}

A useful method to prove the smoothing property  of transition semigroups associated to SDE \eqref{eq4:SDEformalell}
consists in applying, when possible, the so called Bismut-Elworthy-Li formula, introduced
in \cite{Bismut81} and reprised in \cite{ElworthyLi94}
(see also \cite[Lemma\,2.4]{PeszatZabczyk95}, \cite[Lemma\,7.7.3]{DaPratoZabczyk02}, \cite{FuhrmanTessitore02-sto} for the version used here,
and, for a generalization to the non-linear superquadratic case, \cite{Masiero14}).


The Bismut-Elworthy-Li formula have been used to prove
smoothing properties of transition semigroups in three important cases:
\begin{enumerate}
  \item stochastic Burgers and Navier Stokes equations
  (see \cite{DaPratoDebussche98,DaPratoDebussche00});
  \item stochastic reaction-diffusion equations (see \cite{Cerrai99}
  and \cite[Ch.\,6-7]{Cerrai01book});
  \item SDE with invertible diffusion coefficient
  (see e.g. \cite{PeszatZabczyk95} and, later, \cite{FuhrmanTessitore02-sto}
  and \cite{Masiero14} in more general cases).
\end{enumerate}

Here we present the third case,
referring to \cite{FuhrmanTessitore02-sto} for the proofs.

In the probabilistic framework of the previous section, we consider the following assumptions on the data $A,b,\sigma$ in \eqref{eq4:SDEformalell}.



\begin{Assumption}
\label{hp4:Bismutprel}
\begin{enumerate}[(i)]
\item[]
\item The coefficients $A$, $b$, and $\sigma$ of SDE \eqref{eq4:SDEformalell} satisfy Assumption \ref{ch1:ipotesiunobis}
with  $f(s)=\kappa_0s^{-\gamma}$ for some $\kappa_0>0$ and $\gamma\in[0,1/2)$.
  \item $\sigma\in B_b(H, \call\left(  \Xi,H\right))$.
\item
 $b\in\calg^1(H)$ and
 $e^{sA}\sigma(\cdot)\in
\calg^1(H,\call_2(\Xi;H))$ for every $s>0$.
\end{enumerate}
\end{Assumption}

\begin{Proposition}
\label{pr4:FTBismutderX}
Let Assumption \ref{hp4:Bismutprel} hold.
Then, for every $p\geq 2$
the following results hold.
\begin{itemize}
  \item[(i)] The map $x \mapsto X(\cdot, x)|_{[0,T]}$ belongs to
  $\mathcal{G}^{1}( H, \mathcal{K}_\calp^{p,T}(H))$ for every $T>0$.
  \item[(ii)]
For every direction
$h \in H$, the directional derivative process
$\nabla_xX(\cdot,x )h$ (recall that $\nabla$ denotes the Gateaux derivative) is a mild solution
to the SDE
$$
dY(s) = [A Y(s) +
 \nabla_x  b(X(s,x ))Y(s)] ds
+
  \nabla_x \sigma(X(s, x ))Y(s) d W(s),
$$
\item[(iii)]
There exists $C_p>0$ and $\alpha_p\in\R$ such that
$$
|\nabla_x X(\cdot,x )h|_{ \mathcal{K}_\calp^{p,T}(H)}\le C_pe^{\alpha_p T}|h|, \ \ \forall h\in H, \ \forall T>0.
$$
\end{itemize}
\end{Proposition}

\begin{proof}
See, e.g., \cite[Prop.3.3]{FuhrmanTessitore02-ann}, \cite[Prop.\,4.3]{FuTe-ell}. The exponential dependence in time of the estimate of item (iii) can be proved by induction exploiting the time homogeneity of the system.
%
\end{proof}

For the Bismut-Elworthy-Li formula and, consequently, for
the required smoothing property, we also need the following assumption
(see \cite{PeszatZabczyk95}).

\begin{Assumption}
\label{hp4:Bismut} The operator  $\sigma(x)\in\call (\Xi,H)$ is invertible for each $x\in H$ and there exists $C_0>0$ such that  $|\sigma(x)^{-1}|_{\call(H,\Xi)}\leq C_0$.
\end{Assumption}

\begin{Theorem}
\label{th4:FTBismut}
Let Assumptions \ref{hp4:Bismutprel} and
\ref{hp4:Bismut} hold,  let $m\geq 0$, and let $(P_t)_{t\geq 0}$ be the family of linear operators defined through \ref{trsemm}. The following statements hold true.

\begin{enumerate}[(i)]
\item
$P_{s} [\phi]\in \calg^{1}_m(H)$ for every $s>0$ and $\phi\in C_m(H)$.
\item There exists constants $C > 0$ and $a \ge 0$
such that
\begin{equation}
|\nabla P_{s} [\phi] (x)| \leq \frac {Ce^{a s}}{s^{1/2}}
| \phi |_{C_m(H)} (1+|x|^m), \quad \forall s>0, \ \forall x \in H, \ \forall \phi\in C_m(H).
 \label{eq4:stimaderBismut}
\end{equation}
\item We have the representation formula (Bismut-Elworthy-Li formula)
\begin{equation}
\<\nabla P_{s} [\phi] (x),h\> =\E\left[\phi (X(s,x))U^h(s,x)\right],
 \quad\ \forall s>0, \ \forall x,h \in H, \ \forall \phi\in C_m(H),
 \label{eq4:reprderBismut}
\end{equation}
where
\begin{equation}
U^h(s,x):=\frac{1}{s}
\int_0^s \< \sigma (X(\tau,x))^{-1} \nabla_x X(\tau,x)h,dW(\tau)\>_{\Xi}.
 \label{eq4:defUhBismut}
\end{equation}
\end{enumerate}
\end{Theorem}

\begin{proof}
See    \cite[Th.\,4.2]{FuhrmanTessitore02-sto}.
The exponential dependence in time of the estimate \myref{eq4:stimaderBismut} can be proved by induction exploiting the time homogeneity of the system.
%
\end{proof}

%
%
%

\begin{Corollary}
\label{pr4:contPts}
Let Assumptions \ref{hp4:Bismutprel} and
\ref{hp4:Bismut} hold and let $(P_t)_{t\geq 0}$ be the family of linear operators defined through \ref{trsemm}.
Then $(P_t)_{t\geq 0}$ satisfies Assumption \ref{hp4:Psemigroupell} with $K=H$ and $G=I$ for every $m\geq 0$.
\end{Corollary}
\begin{proof}
As we have already observed, the conditions of Assumption \ref{hp4:Psemigroupell}(i)--(iii) are verified as a straightforward consequence of  Theorem \ref{th1:exSDEmult}.

Combining the assumption that $t\mapsto |\Gamma_G(t)|_{\call(K,H)}$ belongs to $\cali$ and Theorem \ref{th4:FTBismut},  the conditions of Assumptions \ref{hp4:Psemigroupell}
(\ref{hp4:smoothingGell}) and (\ref{ip4}) follow.

Finally, Assumption \ref{hp4:Psemigroupell}(\ref{ip5}), that is the fact that the map
 $(0,+\infty)\times H \to H, \ (s,x) \mapsto \nabla^G P_{s}[\phi](x)$ is measurable, directly follows from the representation formula \ref{eq4:reprderBismut}  and using Pettis measurability's Theorem \cite[Th.\,1.1]{Pettis38}, as $H$ is separable.
\end{proof}

\section{Application to stochastic control problems}\label{sec:SOC}
Let $H,K, \Xi$ be separable Hilbert spaces, let $\Lambda$ be a Polish space, let  $G:H\to\call_u(K,H)$, and let ${ L}:H\times\Lambda\to K$.
Let $W$ be a cylindrical Brownian motion in $\Xi$ and consider the following controlled SDE:
\begin{equation}
\begin{cases}
dX(s)=\left[AX(s) + b(X(s)) + G(X(s))L(X(s),a(s))\right]ds
+ \sigma(X(s))dW(s),\\
\label{eq4:SE1ell}
X(0)=x,
\end{cases}
\end{equation}
where the control process $a(\cdot)$ lies in the space {$\calu$  of progressively measurable $\Lambda$-valued processes.}
Let us consider the following assumptions.

\begin{Assumption}
	\label{hp4:b1VTell}
	\begin{enumerate}
	\item[]
\item[(i)] $A,b,\sigma$ satisfy Assumption \ref{ch1:ipotesiunobis}.

\item[(ii)]
For all $x \in H$, $e^{sA}G(x)$ can be extended to $\overline{e^{sA}G(x)}\in \call(K,H)$ for every $s>0$. Moreover, for all $k \in K$ the map $(s,x)\to \overline{e^{sA}G(x)}k$ is measurable and there exists $f_G\in\cali$ such that
$$
 |{\overline{e^{sA}G(x)}}|_{\call(K,H)}\le f_G(s)(1+|x|), \ \ \forall s>0, \ \forall x\in H.
$$

\item[(iii)]
$L\in  B _b(H\times \Lambda ,K)$ and
there exists $f_{GL}\in \cali$ such that
$$
|{\overline{e^{sA}G(x)}L(x,a)- \overline{e^{sA}G(y)}L(y,a)}| \le f_{GL}(s)|x - y |, \ \ \forall s{>} 0, \ \forall x,y\in H, \ \forall a\in\Lambda.
$$
\end{enumerate}
\end{Assumption}
%
%
%
%

Under Hypotehsis \ref{hp4:b1VTell}, for every $x\in H$ and $a(\cdot)\in \calu$,  \eqref{eq4:SE1ell} admits a unique mild solution $X(\cdot;x,a(\cdot))\in \calh_{loc}^p(H)$, for every $p\geq 2$ {{(see \cite[Ch.\,1,\,Sec.\,6]{FGSbook}).}}
Given $\lambda >0$, $x\in H$, and $l:H\times \Lambda\to \R$ measurable, define the functional
\begin{equation}
 	J(x;a(\cdot)) ={\E}\left[ \int_{0}^{+\infty}e^{-\lambda s}l(X(s;x,a(\cdot)),a(s))ds\right], \ \ \ \ a(\cdot)\in \calu.
    \label{eq4:CF1ell}
\end{equation}
The stochastic optimal control problem consists in minimizing the functional above over the set of admissible controls $\calu$, i.e. to solve the optimization problem
\begin{equation}
 		V(x) \coloneqq
 \inf_{a(\cdot)\in \calu } J(x;a(\cdot)), \ \ \ \ x\in H.
 	    \label{eq4:VF1ell}
\end{equation}
The function $V:H\to \R$ is the so called value function of the optimization problem. By standard Dynamic Programming arguments, one formally associates to  this control problem an HJB equation. It reads as
\begin{equation}
\label{eq4:HJBcontrolell}
\displaystyle{ \lambda v(x) -\frac{1}{2}\;
\mbox{\rm Tr}\;[Q(x)D^2v(x)] -\< Ax+b(x),Dv(x)\>- F(x,Dv(x))=0,
\quad x \in H,}
\end{equation}
where $Q(x)=\sigma(x)\sigma^*(x)$ and the Hamiltonian $F$ is defined by
\begin{equation}
\label{eq4:Hamiltoniancontrolell}
	F(x,p) \coloneqq \inf_{a \in \Lambda } F_{CV} (x,p;a), \ \ \ x\in H, \ p\in  H,
\end{equation}
with
\begin{equation}\label{FCV}
F_{CV} (x,p;a)
	\coloneqq
	 \inf_{a \in \Lambda }
\left\{  \< G(x)L(x,a),p \>_H +l(x,a)\right\}, \ \ x\in H, \ a\in \Lambda, \ p\in H.
\end{equation}
It is convenient here
to introduce the modified Hamiltonian $F_0$ as follows
\begin{equation}
\label{eq4:F0controlell}
	F_0(x,q) \coloneqq \inf_{a \in \Lambda }F_{0,CV} (x,q;a), \ \ \ x\in H, \ q\in K,
\end{equation}
where
\begin{equation}\label{F0CV}
{
F_{0,CV} (x,q;a)\coloneqq
 \< L(x,a),q \>_K +l(x,a),  \ \ \ x\in H, \ a\in \Lambda, \ q\in K,}
\end{equation}
and observe that
$$
F(x,p)=F_0(x,G(x)^*p), \ \ \ \forall p\in \cald(G(x)^*),
$$
so \eqref{eq4:HJBcontrolell} can be formally rewritten as
\begin{equation}
\label{eq4:HJBcontrolellbis}
\displaystyle{ \lambda v(x) -\frac{1}{2}\;
\mbox{\rm Tr}\;[Q(x)D^2v(x)] -\< Ax+b(x),Dv(x)\>- F_0(x,D^Gv(x))=0,
\quad x \in H}.
\end{equation}
Denote by $X_0(\cdot,x)$ the unique mild solution of the uncontrolled state equation
\begin{equation}
\begin{cases}
dX(s)=\left[AX(s) + b(X(s))\right]ds
+ \sigma(X(s))dW(s),\\
\label{eq4:SE1ellnew}
X(0)=x,
\end{cases}
\end{equation}
and consider the associated transition semigroup $(P_s)_{s\geq 0}$ defined by $P_s[\phi](x)\coloneqq \E[\phi(X_0(s,x))]$.
Then, with these specifications of $(P_s)_{s\geq 0}$ and $F_0$,   one can apply to HJB \eqref{eq4:HJBcontrolellbis} the results of the latter section to establish existence, uniqueness and regularity of the mild solution if all the assumptions are fulfilled.

\begin{Remark}\label{rem:OFM}
The $G$-regularity of the mild solution to \eqref{eq4:HJBcontrolellbis} is particularly meaningful from the point of view of the control problem,  as it allows to define, in classical sense (unlike viscosity solution do), an optimal feedback map for the problem. In our case, assuming that the  infimum in \eqref{eq4:F0controlell} is obtained by a unique minimum point, it  reads as
\begin{equation}\label{OFM}
a^*(x,D^Gv(x))\coloneqq \textsl{arg\,inf}_{a\in \Lambda}\,\, F_{0,CV} (x,D^Gv(x);a).
\end{equation}
This fact represents an important starting point towards the solution of the problem through a verification theorem. To prove such a result one needs to apply
stochastic calculus and It\^o's formula to exploit the HJB equation and, in this regard, the concept of mild solution does not provide sufficient regularity. So, this issue still needs some work.  A possible strategy to tackle the problem is to prove that the mild solution enjoys the property of being a \emph{strong solution}, i.e., roughly speaking, to be the limit, in some suitable sense, of very regular solutions of approximating equations.  Then, one can argue by approximation to prove the verification theorem (see, e.g., \cite{Gozzi95}). {However, as we will show  in a forthcoming companion paper, this passage is not needed (at least) in the case of optimal control of Ornstein-Uhlenbeck process. Already the notion of mild solution suffices to prove a verification theorem}.
\end{Remark}

\subsection{Neumann boundary control of  stochastic heat  equation with additive noise}
We consider  the optimal control of a nonlinear stochastic heat equation in a given space region ${\mathcal{O}}\subseteq \R^N$ when the control can be exercised only at the boundary of  ${\mathcal{O}}$ or in a subset of ${\mathcal{O}}$.
Precisely we consider the cases when the control at the boundary enters through a Neumann-type boundary condition, corresponding to control the heat flow at the boundary.


\subsubsection{Informal setting of the problem}
\label{SSSE2:HEATBOUNDARYSETTINGD}

Let ${\mathcal{O}}$ be an open, connected, bounded subset of  ${\mathbb{R} }^N$ with regular (in the sense of \cite[Sec.\,6]{Lasiecka80}) boundary $\partial {\mathcal{O}}$. We consider the controlled dynamical system driven by the following SPDE on the time interval $[0,+\infty)$:
\begin{equation}
\label{eq2:state-boundaryPDENeu}
\left\{
\begin{array}{ll}
\displaystyle{\frac{\partial}{\partial s} y(s,\xi) =
\Delta y(s,\xi)+ \sigma \dot{W}(s,\xi)},
& (s,\xi) \in [0,+\infty) \times {\mathcal{O}},
\\\\
y(0,\xi ) = x (\xi), & \xi\in {\overline{\mathcal{O}}}, \\\\
\displaystyle{ \frac{\partial y(s,\xi)}{\partial n}= {\gamma_0} (s,\xi),}
& (s,\xi)\in  [0,+\infty)\times \partial {\mathcal{O}}, \end{array}\right.
\end{equation}
where:
\begin{itemize}
   \item $y: [0,+\infty)\times \mathcal{O} \times \Omega\to \R$ is a stochastic process describing  the evolution of the temperature distribution and is the {\em state variable} of the system;
   \item ${{\gamma_0}}:[0,+\infty)\times \partial \mathcal{O}\times \Omega\to \R$ is a stochastic process representing the heat flow at the boundary; it is the {\em control variable} of the system and acts at the boundary of it: this is the reason of the terminology ``boundary control";
\item $n$ is the outward unit normal vector at the boundary $\partial \calo$;
  \item $x\in L^2(\mathcal{O}) $ is the initial state (initial temperature distribution) in the region $\mathcal{O}$;
  \item $W$ is a cylindrical Wiener process in $L^2(\calo)$;
  \item $\sigma\in \call(L^2(\calo))$.

\end{itemize}

Assume that this equation is well posed (in some suitable sense,
see below for the precise setting) for every given ${\gamma_0(\cdot,\cdot)}$ in a suitable set of admissible control processes ${\calu_0}$
and denote its unique solution  by ${y^{x,\gamma_0(\cdot,\cdot)}}$ to underline the dependence of the state $y$ on the control ${\gamma_0(\cdot,\cdot)}$ and on the initial datum $x$.
The controller aims at minimizing over the set {$\calu_0$} the functional
\begin{equation}
\label{eq1:CFLapDir}
I(x;{\gamma_0(\cdot,\cdot)}) ={\mathbb{E}}\Bigg[ \int_{0}^{+\infty}e^{-\lambda s}\left(\int_{{\mathcal{O}}}\beta_1(y^{x,{\gamma_0(\cdot,\cdot)}}(s,\xi )) d\xi + \int_{{\partial\mathcal{O}}} \! \beta_2 ({\gamma_0} (s,\xi)) d\xi \right)\ud s   \Bigg],
\end{equation}
where $\beta_1, \beta_2:\R\to\R$ are given measurable functions and $\lambda>0$ is a discount factor.

\subsubsection{Infinite dimensional formulation}

We now rewrite the state equation  (\ref{eq2:state-boundaryPDENeu})) and the functional \eqref{eq1:CFLapDir} in an infinite dimensional setting in the space  $H\coloneqq L^2({\mathcal{O}})$. For more details, we refer to \cite[pp.431 ff.]{BDDM07}, \cite[Sec.\,3.3]{LasieckaTriggiani00}, and \cite{Lasiecka80} in a deterministic framework); to \cite{Gozzi02} and \cite[Appendix C]{FGSbook} in a stochastic framework.

Consider the realization of the Laplace operator with vanishing Neumann boundary conditions:\footnote{To be precise, $\cald(A_N)$ is the closure in $H^2(\calo)$ of the set of functions $\phi\in C^2(\overline\calo)$ having vanishing normal derivative at the boundary $\partial\calo$.}
\begin{equation}
\label{eq:def-AN}
\left \{
\begin{array}{l}
\cald(A_N) := \left \{ \phi\in H^2 ({\mathcal{O}})  : \  \frac{\partial \phi}{\partial n} = 0 \; \text{\rm\;  on } \partial{\mathcal{O}}  \right \}, \\[7pt]
A_N\phi:= \Delta\phi, \ \ \forall \phi\in \cald(A_N).
\end{array}
\right.
\end{equation}
It is well-known (see, e.g.,  \cite[Ch.\,3]{Lunardi95}) that $A_N$ is  generates a strongly continuous analytic  semigroup $(e^{t A_N})_{t\geq 0}$ in $H$. Moreover, $A_N$ is a self-adjoint and dissipative operator. In particular $(0,+\infty)\subset \varrho(A_N)$, where $ \varrho(A_N)$ denotes the resolvent set of $A_N$. So,  if $\delta > 0$, then  $(\delta I-A_N)$ is invertible and $(\delta I-A_N)^{-1}\in \mathcal{L}(H)$. 
{Moreover (see, e.g., \cite[App.\,B]{Lasiecka80})
the operator $(\delta I-A_N)^{-1}$ is compact and, consequently, there exists an orthonormal complete sequence $\{e_k\}_{k\in \N}$ such that the operator
$A_N$ is diagonal with respect to it:
\begin{equation}\label{diag}
A_Ne_k=-\mu_k e_k, \ \ k\in\N,
\end{equation}
for a suitable sequence of eigenvalues $\{\mu_k\}_{k\in \N}$ with $\mu_k\ge 0$ (the sign being due to dissipativity of $A_N$). We assume that such sequence is increasingly ordered.
\cite[App.\,B]{Lasiecka80} provides also a growth rate for the sequence of eigenvalues; indeed
\begin{equation}
\label{kj}
\mu_k \sim k^{2/d}.
\end{equation}
Note that \eqref{kj} in particular yields
\begin{equation}\label{kj2}
\exists k_0\in\N: \ \mu_k>0 \ \forall k\geq k_0.
\end{equation}
We have (see again, e.g., \cite[App.\,B]{Lasiecka80})} 
the isomorphic identification (see, e.g.,   \cite[App.\,B]{Lasiecka80})
\begin{equation}\label{pppd}
\cald((\delta I-A_N)^\alpha)=H^{2\alpha}({\mathcal{O}}),\quad \forall \alpha\in \left(0,\frac{3}{4}\right), \  \forall \delta>0.
\end{equation}
{where $H^{s}(\calo)$ denotes the Sobolev space of exponent $s\in \R$.}
Next,
consider the following problem with Neumann boundary condition:
\begin{equation}
\label{eq:Neumann-problem}
\left \{
\begin{array}{ll}
\Delta w(\xi)=\delta w(\xi), & \xi\in {\mathcal{O}} \\[8pt]
\frac{\partial}{\partial n}w(\xi) = \eta(\xi), & \xi\in \partial{\mathcal{O}}.
\end{array}
\right .
\end{equation}
 Given any $\delta> 0$ and $\eta\in L^2(\partial {\mathcal{O}})$, there exists a unique solution $N_\delta\eta\in H^{3/2}({\mathcal{O}})$ to (\ref{eq:Neumann-problem}). Moreover, the operator (\emph{Neumann map})
\begin{equation}
\label{eqapp:defNeumannmap}
N_\delta :  L^2(\partial {\mathcal{O}}) \to H^{3/2}(\calo)
\end{equation}
is continuous (see \cite[Th.\,7.4]{LionsMagenes-1-EN}).
So, in view of \eqref{pppd}, the map
\begin{equation}
\label{eqapp:defNeumannmapdelta}
N_\delta :  L^2(\partial {\mathcal{O}}) \to  \cald((\delta I -A_N)^{\frac{3}{4} - \varepsilon}), \ \ \ \varepsilon\in(0,3/4)
\end{equation}
is continuous.
Consider  the following problem:
\begin{equation}
\label{eq:parabolic-Neumann}
\left \{
\begin{array}{ll}
\frac{\partial}{\partial s} y(s,\xi) = \Delta  y(s,\xi)  + f(\xi), & (s,\xi)\in [0,+\infty) \times {\mathcal{O}}, \\[8pt]
\frac{\partial}{\partial n} y(s,\xi) = {\gamma_0(s,\xi)}, & (s,\xi)\in [0,+\infty)\times \partial{\mathcal{O}}, \\[8pt]
y(0,\xi) = x(\xi), & \xi\in \calo,
\end{array}
\right .
\end{equation}
where {$x,f\in C(\overline\calo;\R)$, $\gamma_0 \in C([0,+\infty)\times \partial\calo;\R)$}.
If $y\in C^{1,2}([0,+\infty)\times\overline{\mathcal{O}})$  solves (\ref{eq:parabolic-Neumann})  in a classical sense
then  $X(s):=y(s,\cdot)$ solves an evolution equation in $H$.
We sketch the argument (see also, e.g.,  \cite[Sec.\,13.2]{DaPratoZabczyk96}).


{As $y$ is smooth, by classical theory $N_\delta \gamma_0$ is smooth too and
$\frac{\partial }{\partial s} N_\delta \gamma_0=N_\delta  \frac{\partial}{\partial s} \gamma_0$ for all $s\geq 0, \xi\in \calo$. Setting $\gamma(s)(\xi):=\gamma_0(s,\xi)$ and considering $\gamma$ as an element of $L^2_{loc}([0,+\infty);L^2(\partial \calo,\R)$, we then have that the map
$s\mapsto z(s):= X(s) - N_\delta\gamma(s)
$
belongs to $C^1([0,+\infty),H)\cap C([0,+\infty), \cald(A_N))$
 and}
\begin{eqnarray*}
z'(s) &=&X'(s) - \frac{\ud }{\ud s}N_\delta \gamma(s)
=\Delta X(s) + f - N_\delta  \gamma'(s)\\
&=& \Delta (X(s)- N_\delta\gamma(s)) + f +\Delta N_\delta\gamma(s)- N_\delta \gamma'(s)
\\&=& \Delta (X(s)- N_\delta\gamma(s)) + f +\delta N_\delta\gamma(s)- N_\delta \gamma'(s)\\
&=& A_N z(s)+f  +\delta N_\delta\gamma(s) - N_\delta \gamma'(s).
\end{eqnarray*}
In particular,  $z$ is a strict solution (see \cite[Def.\,4.1.1]{Lunardi95}) to the evolution equation in $H$
\begin{equation}
\label{eq:per-z}
\left \{
\begin{array}{l}
z'(s) = A_N z(s)+f  +\delta N_\delta  \gamma(s) - N_\delta \frac{\ud}{\ud s}\gamma(s), \ \ \ \forall s\geq 0,
\vspace{5pt}\\
z(0) = x- N_\delta \gamma(0).
\end{array}
\right .
\end{equation}
It can be written in mild form as
\begin{equation}\nonumber
z(s) = e^{s A_N} \left(x- N_\delta\gamma(0)\right) + \int_0^s e^{(s-r) A_N} \left [  f +\delta N_\delta \gamma(r)- N_\delta  \gamma'(r) \right ] \ud r.
\end{equation}
Therefore
$$
X(s) = z(s) + N_\delta\gamma(s)
= e^{s A_N} \left(x- N_\delta\gamma(0)\right) + \int_0^s e^{(s-r) A_N} \left [  f +\delta N_\delta \gamma(r)- N_\delta  \gamma'(r) \right ] \ud r + N_\delta \gamma(s), \ \ \ \ \forall s\geq 0.
$$
Now, an application of  \cite[Lemma\,13.2.2]{DaPratoZabczyk96} {(integration by parts)} yields
\[X(s)=e^{s A_N} x - A_N \int_0^{s} e^{(s-r)A_N} N_\delta \gamma(r)\ud r + \int_0^s e^{(s-r)A_N} (f+\delta N_\delta \gamma(r))\ud r, \ \ \ \ \forall  s\geq 0,
\]
i.e.
\begin{equation}
\label{eq:sol-parabolic-Neumann-mild}
X(s) = e^{s A_N} x - (A_N-\delta I) \int_0^{s} e^{(s-r) A_N} N_\delta \gamma(r)\ud r + \int_0^s e^{(s-r)A_N} f\ud r, \ \ \ \forall s\geq 0.
\end{equation}
%
It is suitable, for our purpose, to rewrite \eqref{eq:sol-parabolic-Neumann-mild} further, by exploiting as much  as possible the regularity of the map $N_\delta$. Indeed, in view of the continuity
of the map in \eqref{eqapp:defNeumannmapdelta}, we have, for every $\varepsilon>0$,
\begin{equation}
\label{eq:appA-grazieaBoachner-NEU}
- (A_N-\delta I) \int_0^{s} e^{(s-r) A_N} N_\delta\gamma(r)\ud r
=\int_0^{s} (\delta I-A_N)^{\frac{1}{4}+\varepsilon} e^{(s-r)A_N} L_N^{\delta,\varepsilon}\gamma(r)\ud r,
\end{equation}
where
$
L_N^{\delta,\varepsilon}:= (\delta I-A_N)^{\frac{3}{4}-\varepsilon} N_\delta \in \mathcal{L}(L^2(\partial\calo), H)$.
Therefore, we can rewrite (\ref{eq:sol-parabolic-Neumann-mild}) as
\begin{equation}
\label{eq:sol-parabolic-Neumann-mild1}
X(s) = e^{s A_N} x +\int_0^{s} (\delta I-A_N)^{\frac{1}{4}+\varepsilon} e^{(s-r)A_N} L_N^{\delta,\varepsilon}\gamma(r)\ud r + \int_0^s e^{(s-r)A_N} f(r)\ud r.
\end{equation}
Therefore, setting $G_N^{\delta,\varepsilon}\coloneqq  (\delta I-A_N)^{\frac{1}{4}+\varepsilon}$, it is meaningful to rephrase  \eqref{eq2:state-boundaryPDENeu} in $H$ as
\begin{equation}
\label{eq:strongformAN}
\left \{
\begin{array}{l}
\ud X(s) = \left [  A_N X(s) + G_N^{\delta,\varepsilon} L_N^{\delta,\varepsilon}\gamma(s) \right ] \ud s
+ \sigma \ud W(s), \\[8pt]
X(0) = x.
\end{array}
\right.
\end{equation}
 We are now in the framework of \eqref{eq4:SE1ellnew}, with $K=H$,
 {$U=L^2(\partial \mathcal{O})$}, $A=A_N$, $b\equiv 0$, $G(\cdot)\equiv G_N^{\delta,\varepsilon}$, $L(\cdot,a)\equiv L_N^{\delta,\varepsilon}$, $a=\gamma$, and, with a slight abuse of notation,  $\sigma(\cdot)\equiv \sigma$.
Let us consider, as set of admissible controls,
$$
\calu\coloneqq \left\{\gamma: [0,+\infty)\times \Omega \to \Lambda: \hbox{ $\gamma(\cdot)$ is
$(\mathscr{F}_s)$-progressively measurable}
\right\}.
$$
where $\Lambda$ is a bounded closed subset of $L^2(\partial \calo)$.
Note that $L_N^{\delta,\varepsilon}(\Lambda)$ is bounded in $H$
and that, since the semigroup $\{e^{tA_N}\}_{t \ge 0}$ is strongly
continuous and analytic, the operator
$e^{sA_N} G_N^{\delta,\varepsilon}$ can be extended to  $\overline{e^{sA_N} G_N^{\delta,\varepsilon}}=G_N^{\delta,\varepsilon}e^{sA_N}\in \call(H)$ for every $s>0$ and
\begin{equation}\label{Assq}
\bigg|\overline{e^{sA_N} G_N^{\delta,\varepsilon}}\bigg|_{\call(H)}\leq
Cs^{\frac14 + \varepsilon}, \qquad \forall s > 0
\end{equation}
(see e.g. \cite{Pazy83}, Theorem 6.13-(c)).

Now, assume the following.
\begin{itemize}
\item[(A1)] {$\sigma$} satisfies Assumption \ref{hp4:ABQforOU}(iii).
 \item[(A2)] Letting  $Q_s$ be the operator defined in Assumption \ref{hp4:ABQforOU}(ii), we have
 \begin{equation}\label{asss}
\overline{e^{sA_N}G_N^{\delta,\varepsilon}}(H)\subseteq Q_{s}^{1/2}(H), \qquad \forall s > 0.
\end{equation}
\item[(A3)] The map $t\mapsto |\Gamma_{G_N^{\delta,\varepsilon}}(t)|_{\call(K,H)}$ belongs to $\cali$.
\end{itemize}
We will discuss in the next example conditions guaranteeing (A1)--(A3).
Under such assumptions, Assumption \ref{hp4:b1VTell} is verified; it follows that,   for every $\gamma\in \calu$, there exists a unique mild solution  $X(\cdot,x,\gamma(\cdot))\in \calh_\calp^{p,loc}(H)$ to \eqref{eq:strongformAN} for every $p\geq 2$.  Defining
\[
l\colon H\times \Lambda \to \mathbb{R},  \ \ l(x,{\eta}) \coloneqq \int_{{\mathcal{O}}}\beta_1 (x(\xi)) \ud\xi+\int_{{\partial\mathcal{O}}}\beta_2 ({\eta}(\xi )) \ud\xi,
\]
the functional  (\ref{eq1:CFLapDir}) can be rewritten in the Hilbert space setting as
\begin{equation}
\label{eq1:CFLapDir-restatedbis}
J(x;{\gamma}(\cdot)) ={\mathbb{E}}\Bigg[ \int_{0}^{+\infty} e^{-\lambda s}
l (X(s;x,{\gamma}(\cdot)), \gamma(s)) ds \Bigg].
\end{equation}
Setting $Q=\sigma\sigma^*$, the HJB equation associated to the minimization of \eqref{eq1:CFLapDir-restatedbis} is
\begin{equation}
\label{eq2:HJDirichletNeumann}
\displaystyle{\lambda v(x) - \frac{1}{2}\;
\mbox{\rm Tr}\;[QD^2v(x)]
- \left\langle A_Nx,Dv(x) \right\rangle-  \inf_{{\eta} \in \Lambda} \left\{\left\langle  L_N^{\delta, \varepsilon} {\eta},D^{G_N^{\delta,\varepsilon}} v(x) \right\rangle +l(x,{\eta})  \right\}=0.}
\end{equation}
Hence, if $\beta_1,\beta_2$ satisfy proper continuity and growth assumptions guaranteeing that
$$F_0(x,y,z)\coloneqq   \inf_{{\eta} \in \Lambda} \left\{\left\langle  {L_N^{\delta, \varepsilon}\eta},z \right\rangle +l(x,\eta)  \right\}$$  satisfies Assumptions \ref{hp4:F0ellbis}, we can apply Corollary \ref{corcor} to this problem.

\begin{Example}\label{ex2}
(cf.   \cite[Ex.\,6.3]{Gozzi95} and \cite[Ex.\,13.1.2]{DaPratoZabczyk02}).
Here we discuss the validity of (A1)--(A3) above in a specific example. Within the framework above,
let $\calo= (0,\pi)\subset \R$.
The Laplace operator $A_N$ has a diagonal representation on $H=L^2(0,\pi)$.
The orthonormal basis of eigenfunctions is
$$
e_{n}(\xi ) =
\cos (n\xi),\quad \xi\in(0,\pi), \ n\in\N,
$$
and
\begin{equation*}
\label{autlapN}
A_N e_n =-\alpha_ne_n, \ \ n\in \N,
\end{equation*}
where $\alpha_n=n^2$.
In particular, looking at \eqref{eq4:Qtdiagnuclear}, we see that every ${{\sigma}}\in\call(H)$ satisfies (A1). Consider the case ${{\sigma}}=I$. Take $\delta>0$ and set $\beta\coloneqq \frac{1}{4}+\varepsilon$. 	{{For every $n\in\N$, we have $G_N^{\delta,\varepsilon}e_n=(\delta I-A_N)^\beta e_n= g_ne_n$, where
$g_n\coloneqq \left(\delta+n^2\right)^\beta$.}}
Moreover,
\begin{equation}\label{eq:Gammanormest}
\sup_{n\ge 1}
\frac{2n^2(\delta+n^2)^{2\beta}}{e^{2t n^2}-1}\le {\sup_{n\ge 1}}
\frac{2(1+\delta)^{2\beta}(n^2)^{1+2\beta}}{e^{2t n^2}-1}\le \ \frac{C_0}{t^{1+2\beta}}
\end{equation}
where $C_0\coloneqq {{2(1+\delta)^{2\beta} \sup_{s>0} \frac{s^{1+2\beta}}{e^{2s}-1}}}<+\infty$.
Then, recalling what said in Example \ref{ex1}, in particular in \myref{eq4:diagNC}-\eqref{eq4:diagGammaGint}, we have
\[
|\Gamma_G(t)|^2_{\call(H)}\le
\frac{C_0}{t^{1+2\beta}},
\]
so \myref{eq4:diagGammaGint} holds with
$\gamma_G(t)=\frac{C_0^{1/2}}{t^{\frac{1}{2 }+\beta}}$.
This shows that, {{if $\varepsilon\in(0,1/4)$, then}}  (A2)--(A3) hold.

\end{Example}

\begin{Remark}
\begin{enumerate}[(i)]
\item The example given can be extended to include the case of optimal boundary control of the reaction-diffusion SPDE
$$
{\frac{\partial}{\partial s} y(s,\xi) =
\Delta y(s,\xi)+ f(y(s,\xi))+\sigma \dot{W}(s,\xi)}.
$$
where $f$ satisfies suitable dissipativity conditions.
This problem is studied in \cite{Cerrai01-40} in the case of distributed control. Our techniques allows to treat the same case with control of Neumann type at the boundary.

\item
 One can deal with the multidimensional extension of Example \ref{ex2}, when $\calo =(0,\pi)^2$ (see \cite[Rem.\,6.4]{Lasiecka80}).
{ In this case $\alpha_n \sim n$, hence $\sum_{n=1}^{+\infty}\frac{1}{\alpha_n}=+\infty$; so, in order
to get (A1) satisfied, the diffusion coefficient ${\sigma}$ cannot be the identity. Instead,  assume
that {$Q:=\sigma \sigma^*$} is diagonal with eigenvalues
$q_n \sim n^{-\theta +1}$, $\theta>1$.
Then
\begin{align*}
&\sup_{n\in\mathbb{N}}{\frac{2n(\delta+n)^{2\beta}}{(e^{2tn}-1)q_n}}\sim \sup_{n\in\mathbb{N}}{\frac{2n^{\theta}(\delta+n)^{2\beta}}{e^{2tn}-1}}
\\&\le  \sup_{n\in\mathbb{N}}  \frac{2(1+\delta)^{2\beta}n^{\theta +2\beta}}{e^{2t n}-1} =    \sup_{n\in\mathbb{N}}  \frac{2(1+\delta)^{2\beta}\left(tn\right)^{\theta  +2\beta}}{\left(e^{2t n}-1\right)t^{\theta  +2\beta}} \leq \frac{C_0}{t^{\theta  +2\beta}},
\end{align*}
where $C_0\coloneqq 2(1+\delta)^{2\beta} \sup_{s>0} \frac{s^{\theta  +2\beta}}{e^{2s}-1}$. The latter is finite if and only if ${\theta  +2\beta}>1$, which is the case, as we are taking $\theta>1$ and $\beta\in(1/4,1/2)$. So, (A1) is satisfied.
Then, to have (A2) and (A3) satisfied too, we must take $\theta>1$ such that $\theta  +2\beta<2$, i.e. $\theta<2(1-\beta)$. This is clearly  possible as $\beta=\frac{1}{4}+\varepsilon\in(1/4,1/2)$.
}
\end{enumerate}

\end{Remark}

\addcontentsline{toc}{section}{References}
\bibliographystyle{plain}
\bibliography{Arxiv-rex.bbl}

\end{document}